\documentclass[11pt]{article}

\usepackage[shortlabels]{enumitem}
\usepackage{amssymb}
\usepackage[mathscr]{euscript}
\usepackage{color}
\usepackage{amsfonts}
\usepackage{graphicx}
\usepackage{latexsym}
\usepackage{amsmath}
\usepackage{changepage}
\usepackage{color}
\usepackage{latexsym}
\usepackage{tablefootnote}
\usepackage{epsfig}
\usepackage{multirow}
\usepackage{float}
\usepackage{array}
\usepackage{bbm}

\def\bea{\begin{eqnarray}}
\def\beaa{\begin{eqnarray*}}
\def\beq{\begin{equation}}
\def\eea{\end{eqnarray}}
\def\eeaa{\end{eqnarray*}}
\def\eeq{\end{equation}}

\allowdisplaybreaks

\begin{document}

\newtheorem{definition}{Definition}[section]
\newtheorem{proposition}{Proposition}[section]
\newtheorem{corollary}{Corollary}[section]
\newtheorem{lemma}{Lemma}[section]
\newtheorem{theorem}{Theorem}[section]
\newtheorem{remark}{Remark}[section]
\newtheorem{assumption}{Assumption}[section]
\newtheorem{example}{Example}[section]
\newtheorem{algorithm}{Algorithm}[section]
\newtheorem{procedure}{Procedure}[section]

\title{Superlinear Convergence of an Interior Point Algorithm on Linear Semi-definite Feasibility Problems}

\author{{\bf{Chee-Khian Sim}}\footnote{Email address: chee-khian.sim@port.ac.uk} \\ School of Mathematics and Physics\\ University of Portsmouth \\ Lion Gate Building, Lion Terrace \\ Portsmouth, PO1 3HF \\ United Kingdom}

\date{August 20, 2024}

\maketitle
\begin{abstract}
{\textcolor{black}{In the literature, besides the assumption of strict complementarity, superlinear convergence of implementable polynomial-time interior point algorithms using known search directions, namely, the HKM direction, its dual or the NT direction, to solve semi-definite programs (SDPs) is shown by (i) assuming that the given SDP is nondegenerate and making modifications to these algorithms \cite{Kojima}, or (ii) considering special classes of SDPs, such as the class of linear semi-definite feasibility problems (LSDFPs) and requiring the initial iterate to the algorithm to satisfy certain conditions \cite{Sim1,Sim6}. Otherwise, these algorithms are not easy to implement even though they are shown to have polynomial iteration complexities and superlinear convergence \cite{Luo}.  The conditions in \cite{Sim1,Sim6} that the initial iterate to the algorithm is required to satisfy to have superlinear convergence when solving LSDFPs however are not practical.  In this paper, we propose a practical initial iterate to an implementable infeasible interior point algorithm that guarantees superlinear convergence when the algorithm is used to solve the homogeneous feasibility model of an LSDFP.}}  


\vspace{10pt}

\end{abstract}


\section{Introduction}\label{sec:introduction}

Many problems in diverse areas, such as optimal control, estimation and signal processing, communications and networks, statistics, and finance, can be modelled well as semi-definite programs (SDPs) \cite{Boyd}.  Finding effective and efficient ways to solve this class of problems is hence practically important.  Interior point methods (IPMs) have been proven to be successful in solving SDPs - see for example \cite{Alizadeh,Kojima1,Monteiro3,Zhang}.  {\textcolor{black}{Research on interior point methods is active with recent papers, such as \cite{Alzalg,Faybusovich,Pougkakiotis,Rigo}\footnote{\textcolor{black}{Note that papers \cite{Alzalg,Faybusovich,Rigo} consider symmetric cone problems, which generalize SDPs.}}, proposing contemporary interior point algorithms to solve SDPs.}}  An important subclass of the class of SDPs is the class of linear semi-definite feasibility problems (LSDFPs), which also have applicability in diverse areas.  This class of problems includes linear matrix inequalities (LMIs).

Among different IPMs, primal-dual path following interior point algorithms are the most successful and most widely studied.  In this paper, we focus on an infeasible predictor-corrector primal-dual path following interior point algorithm to solve the homogeneous feasibility model \cite{Potra} of an LSDFP.  Polynomial iteration complexity of the algorithm to solve a primal-dual SDP pair has been shown in \cite{Potra}.  In this paper, we consider the local convergence behavior of the algorithm.  

It proves not an easy task to show superlinear convergence of an implementable interior point algorithm that has polynomial iteration complexity on an SDP with minimal assumptions on the problem and no modifications to the algorithm.  This is especially so for the HKM (and its dual) or the NT search direction used in these algorithms, although superlinear convergence is quickly established for the AHO search direction \cite{Kojima2, Lu3}, when researchers started focusing their attention on interior point algorithms to solve SDPs. 
 In the literature when the HKM (and its dual) or the NT search direction is used in an interior point algorithm, such as \cite{Kojima}, in addition to strict complementarity assumption, nondegeneracy assumption at an optimal solution and modifications to the algorithm, such as solving the corrector-step linear system in an iteration repeatedly instead of only once (``narrowing" the central path neighborhood\footnote{We note that this idea is used in \cite{Nesterov2} to show superlinear convergence of an interior point algorithm on a wide class of conic optimization problems.  Also, this is related to the sufficient conditions on superlinear convergence behavior of iterates generated by the algorithm  as studied in \cite{Potra2,Potra1}.}), need to be imposed for superlinear convergence of the interior point algorithm.  In \cite{Luo}, without assuming  nondegeneracy, the feasible interior point algorithm considered in the paper is shown to have polynomial iteration complexity and superlinear convergence.  However, the algorithm is not easy to implement. The idea behind the algorithm considered in \cite{Luo} to have superlinear convergence when solving an SDP is to force the centrality measure of the $k^{th}$ iterate to converge to zero as $k$ tends to infinity.  This can be enforced in practice by for example, repeatedly solving the corrector-step linear system in an iteration, as in \cite{Kojima}, so as to ``narrow" the neighborhood of the central path in which iterates lie.    Hence, either a polynomial-time interior point algorithm has to be modified or the algorithm is hard to implement for guaranteed superlinear convergence of these algorithms.  Otherwise, special structure needs to be imposed on the SDP, such as considering linear semi-definite feasibility problems (LSDFPs), for superlinear convergence \cite{Sim1,Sim6}\footnote{In \cite{Sim1,Sim6}, the search direction used in the interior point algorithm considered is the (dual) HKM and the NT direction respectively.}.  For the latter, we also require additionally that a suitable initial iterate {\textcolor{black}{be}} chosen.  However, the required initial iterate to guarantee superlinear convergence given in \cite{Sim1,Sim6} is not practical, {\textcolor{black}{since the conditions required in these papers are not easy to achieve practically}}.  {\textcolor{black}{In this paper, we improve on the results in \cite{Sim1,Sim6} by proposing an initial iterate that can be obtained in practice (by solving a primal-dual SDP pair), and using this initial iterate to an implementable interior point algorithm on the homogeneous feasibility model of the LSDFP, we have superlinear convergence of the algorithm.  The novelty in showing this superlinear convergence result is the nontrivial application of what is known in the literature \cite{Sim1,Sim6} to show the result.  Two key steps that are needed before we can use relevant results in the literature \cite{Sim1,Sim6} are   (i) applying the interior point algorithm on the homogeneous feasibility model of the LSDFP, instead of the LSDFP itself; (ii) suitably reformulating the homogeneous feasibility model of the LSDFP as a semi-definite linear complementary problem (SDLCP).}}


In Section \ref{sec:LSDFPhomogeneousmodel}, we describe the homogeneous feasibility model of an LSDFP, and in Section \ref{sec:SDLCP}, we express the homogeneous feasibility model as a semi-definite linear complementarity problem (SDLCP). {\textcolor{black}{The latter}} allows us to apply results in the literature in Section \ref{sec:interiorpointalgorithm} to show superlinear convergence of an implementable interior point algorithm on the homogeneous feasibility model. We conclude the paper with Section \ref{sec:conclusion}.


\subsection{Notations}\label{subsec:notations}

The space of symmetric $n \times n$ matrices is denoted by $S^n$.  The cone of positive semi-definite (resp., positive definite) symmetric matrices is denoted by $S^n_+$ (resp. $S^n_{++}$).  The identity matrix is denoted by $I_{n \times n}$, where $n$ stands for the size of the matrix.  We omit the subscript when the size of the identity matrix is clear from the context.  The matrix {\textcolor{black}{$E^{i,j} \in \Re^{n_1 \times n_1}$}} is defined to have $1/2$ in its $(i,j)$ and $(j,i)$ {\textcolor{black}{entries}}, and zero everywhere else.

Given a matrix $G \in \Re^{n_1 \times n_2}$, $\| G \|_F: = \sqrt{{\rm{Tr}}(GG^T)}$ is the Frobenius norm of $G$, where ${\rm{Tr}}(\cdot)$ is the trace of a square matrix.  $G_{ij}$ is the entry of $G$ in the $i$th row and the $j$th column of $G$.  On the other hand, given a vector $x \in \Re^n$, $\| x \|$ refers to its Euclidean norm, and $x_i$ is its $i^{th}$ entry.


Given $X \in S^n$, ${\rm{svec}}(X)$ is defined to be
\begin{eqnarray*}
{\rm{svec}}(X) := (X_{11}, \sqrt{2}X_{21}, \ldots, \sqrt{2}X_{n1}, X_{22}, \sqrt{2}X_{32}, \ldots, X_{n-1,n-1}, \sqrt{2} X_{n,n-1}, X_{nn})^T \in \Re^{\tilde{n}},
\end{eqnarray*}
where $\tilde{n} = n(n+1)/2$.  ${\rm{svec}}(\cdot)$ sets up a one-to-one correspondence between $S^n$ and $\Re^{\tilde{n}}$.

{\textcolor{black}{Given a linear map $\mathcal{A}:S^n \rightarrow \Re^{\tilde{n}}$, where $\tilde{n} = n(n+1)/2$, we can view $\mathcal{A}$ as an $\tilde{n} \times \tilde{n}$ matrix, where its $i^{th}$ row is given by  ${\rm{svec}}(\mathcal{A}_i)^T, i = 1, \ldots, \tilde{n}$.  Here, $\mathcal{A}_i$ is a symmetric matrix of size $n$ which can be easily determined from the given map $\mathcal{A}$.}}

\section{A Linear Semi-definite Feasibility Problem and its Homogeneous Feasibility Model}\label{sec:LSDFPhomogeneousmodel}






\noindent Given $C, A_i \in S^n$, $i = 1, \ldots, m$, and $b = (b_1, \ldots, b_m)^T \in \Re^m$.  A (primal) semi-definite program (SDP) is given by
\begin{eqnarray}\label{primalSDP}
\begin{array}{ll}
\min & {\mbox{Tr}}(CX) \\
{\mbox{subject\ to}} & {\mbox{Tr}}(A_i X) = b_i,\ i = 1, \ldots, m,
\\
& X \in S^n_+.
\end{array}
\end{eqnarray}


The dual of (\ref{primalSDP}) is given by
\begin{eqnarray}\label{dualSDP}
\begin{array}{ll}
\max & b^Ty \\
{\mbox{subject\ to}} & \sum_{i=1}^{m} y_iA_i + Y = C, \\
& Y \in S^n_+
\end{array}
\end{eqnarray}
where $y = (y_1, \ldots, y_m)^T \in \Re^m$.  {\textcolor{black}{We consider the case when (\ref{primalSDP}) and (\ref{dualSDP}) are feasible in this paper.}}

%

A linear semi-definite feasibility problem (LSDFP) is a primal-dual SDP pair (\ref{primalSDP})-(\ref{dualSDP}) with $C = 0$ or $b = 0$.  {\textcolor{black}{We call the LSDFP with $C = 0$, the sLSDFP in this paper}}.   Also, we assume that $b_i \not= 0$, for some $i = 1, \ldots, m$, {\textcolor{black}{in the sLSDFP}} to avoid trivial considerations.

\begin{remark}\label{rem:LMI}
{\textcolor{black}{A (nonstrict) linear matrix inequality (LMI) has the form
\begin{eqnarray}\label{LMI}
\sum_{j = 1}^{l} z_j B_j + B_0 \in S^n_{+},
\end{eqnarray}
where $z_j \in \Re, j = 1, \ldots l$, are the decision variables, and $B_j \in S^n, j = 0, \ldots, l$, are given.  Solving LMIs has wide applicability in diverse areas, as discussed for example in \cite{Boyd2}.   It is easy to see that (\ref{LMI}) can be written as the primal SDP (\ref{primalSDP}) with $C = 0$, for suitable $A_i, b_i, i = 1, \ldots, m$, where we wish to find $X \in S^n_{+}$ that solves {\textcolor{black}{this problem}}.  Hence, by solving the primal SDP (\ref{primalSDP}) with $C = 0$, we are able to solve (\ref{LMI})\footnote{It is known that implementable interior point algorithms with polynomial iteration complexities can be used to solve LMIs \cite{Boyd2,Boyd}, but to the best of our knowledge, it is unknown whether these algorithms can solve (nonstrict) LMIs with proven good local convergence behavior, which we show in this paper.} .}}
\end{remark}

\noindent We impose the following assumptions on the {\textcolor{black}{sLSDFP}} throughout this paper:
\begin{assumption}\label{ass:SDP}
\begin{enumerate}[(a)]
\item There exists  $(y, Y) \in \Re^m \times S^n_{++}$ that is feasible to the dual SDP (\ref{dualSDP}) with $C = 0$.  {\textcolor{black}{We call the dual SDP  (\ref{dualSDP}) with $C = 0$, the sDSDP from now onwards.}}
\item $A_1, \ldots, A_m$ are linearly independent.
\end{enumerate}
\end{assumption}  
Assumption \ref{ass:SDP}(b) is without loss of generality.  {\textcolor{black}{On the other hand, Assumption \ref{ass:SDP}(a) ensures that the duality gap between the primal SDP (\ref{primalSDP}) and the dual SDP (\ref{dualSDP}) is zero - see for example \cite{Boyd}.}} 

\begin{remark}\label{rem:findinginitialiterate}
{\textcolor{black}{We can find $(y,Y) \in \Re^m \times S^n_{++}$ in Assumption \ref{ass:SDP}(a) by {\textcolor{black}{solving the following primal-dual SDP pair with a primal-dual path following interior point algorithm\footnote{{\textcolor{black}{In fact, by applying the primal-dual path following interior point algorithm on the primal-dual SDP pair, we will know whether Assumption \ref{ass:SDP}(a) holds by checking the solution the algorithm converges to.}}}}}:
\begin{eqnarray*}\label{primal-dualSDP}
\begin{array}{ll}
\min & 0 \\
{\mbox{subject\ to}} & {\mbox{Tr}}(A_i X) = 0,\ i = 1, \ldots, m,
\\
& X \in S^n_+, \\
& \\
\max & 0 \\
{\mbox{subject\ to}} & \sum_{i=1}^{m} y_iA_i + Y = 0, \\
& Y \in S^n_+.
\end{array}
\end{eqnarray*} }}
{\textcolor{black}{From \cite{Sim6} (see also \cite{Potra1,Sim1}), we have polynomial iteration complexity and superlinear convergence of an infeasible primal-dual path following interior point algorithm on the above primal-dual SDP pair to obtain its solution.  Hence, we are able to get a strictly feasible solution to the sDSDP efficiently and fast. {\textcolor{black}{Extra computational time is certainly needed to generate this stricly feasible solution to the {\textcolor{black}{sDSDP}}.  Whether it is worth the computational effort for this when solving practical problems that can be formulated as sLSDFPs depends on the importance to obtain an accurate solution to the given problem fast.}}}} 
\end{remark}
{\textcolor{black}{In this paper, we apply an infeasible interior point algorithm using the (dual) HKM search direction or the NT search direction to solve the homogeneous feasibility model of the {\textcolor{black}{sLSDFP}}.   Since the algorithm is an interior point algorithm, the initial iterate $(X_0,y_0,Y_0)$ to the algorithm needs to satisfy $X_0, Y_0 \in S^n_{++}$.  To prove superlinear convergence of the algorithm on the homogeneous feasibility model of the {\textcolor{black}{sLSDFP}}, we further require the initial iterate $(X_0,y_0,Y_0)$ to have $(y_0, Y_0)$ feasible to the {\textcolor{black}{sDSDP}}.  Although this is an additional requirement on the initial iterate, it can be satisfied in practice, as discussed in Remark \ref{rem:findinginitialiterate}.  This superlinear convergence result improves on what is known in the literature \cite{Sim1,Sim6}, where the initial iterate, $(X_0,y_0,Y_0)$, is required to satisfy impractical conditions for superlinear convergence. }}



{\textcolor{black}{We now introduce the homogeneous feasibility model of the primal-dual SDP pair (\ref{primalSDP})-(\ref{dualSDP}), which first appeared in \cite{Potra}.  We consider the general setting when $C$ is not required to be zero when describing the model.}}  The model is given by the
following homogeneous system:
\begin{eqnarray}
 & {\mbox{Tr}}(A_i X) =  b_i \tau,\ i = 1, \ldots, m,
\label{homogeneous1} \\
 & \sum_{i=1}^{m} y_i A_i + Y = \tau C,  \label{homogeneous2} \\
 & \kappa = b^T y - {\mbox{Tr}}(CX),  \label{homogeneous3} \\
 & X \in S^n_{+}, y \in \Re^m, Y \in S^n_{+}, \tau \geq 0, \kappa \geq 0.
\label{homogeneous4}
\end{eqnarray}
The model has been incorporated in the software package SeDuMi, according to \cite{de_Klerk}, {\textcolor{black}{to solve SDPs}}, and also in the commercial software package MOSEK, according to \cite{Dahl} (see also \cite{Anderson}), to solve conic optimization problems, which include SDPs.

Observe that (\ref{homogeneous1})-(\ref{homogeneous3}) implies that
\begin{eqnarray}\label{dualitygap}
{\mbox{Tr}}(XY) + \tau \kappa = 0,
\end{eqnarray}
\noindent from which we obtain, using (\ref{homogeneous4}),
\begin{eqnarray*}
XY & = & 0, \\
\tau \kappa & = & 0.
\end{eqnarray*}

Furthermore, observe that a solution to the homogeneous system is readily available and is given by
$(X,y,Y,\tau,\kappa) = (0,0,0,0,0)$.  However, we cannot derive optimal solutions to (\ref{primalSDP}) and (\ref{dualSDP}) from this.  If there exists a
solution $(X^\ast, y^\ast, Y^\ast, \tau^\ast, \kappa^\ast)$ to
(\ref{homogeneous1})-(\ref{homogeneous4}) such that $\kappa^\ast =
0$ and $\tau^\ast > 0$, then $(X^\ast/\tau^\ast, y^\ast/\tau^\ast,
Y^\ast/\tau^\ast)$ is an optimal solution to the primal-dual SDP
pair (\ref{primalSDP})-(\ref{dualSDP}) with zero duality gap.  Conversely, if $(X^\ast, y^\ast, Y^\ast)$ is an optimal solution to the primal-dual SDP pair (\ref{primalSDP})-(\ref{dualSDP}) with zero duality gap, then $(X^\ast, y^\ast, Y^\ast, 1, 0)$ solves the system (\ref{homogeneous1})-(\ref{homogeneous4}).   Interior point algorithms when applied to the homogeneous feasibility model can be used to find optimal solutions to the  primal-dual SDP pair (\ref{primalSDP})-(\ref{dualSDP}) by finding a solution $(X^\ast, y^\ast, Y^\ast, \tau^\ast, \kappa^\ast)$ to
(\ref{homogeneous1})-(\ref{homogeneous4}) such that $\kappa^\ast =
0$ and $\tau^\ast > 0$.  Such an interior point algorithm is discussed in \cite{Potra}, which is also given in Section \ref{sec:interiorpointalgorithm} below where it is specialized to the case when $C = 0$.  These interior point algorithms necessarily have to be of the infeasible type in that the initial iterate and subsequent iterates generated by the algorithm {\textcolor{black}{do not always}} satisfy (\ref{homogeneous1})-(\ref{homogeneous3}).  This is so because if an iterate $(X_k, y_k, Y_k, \tau_k, \kappa_k) \in S^n_{++} \times \Re^m \times S^n_{++} \times \Re_{++} \times \Re_{++}$ for some $k \geq 0$ satisfies (\ref{homogeneous1})-(\ref{homogeneous3}), then (\ref{dualitygap}) holds, which is impossible since $X_k, Y_k \in S^n_{++}$ and $\tau_k, \kappa_k > 0$.
{\textcolor{black}{
\begin{remark}\label{rem:ourcase}
{\textcolor{black}{When $C = 0$ in (\ref{homogeneous1})-(\ref{homogeneous4}), they}} are given by:
\begin{eqnarray}
 & {\mbox{Tr}}(A_i X) =  b_i \tau,\ i = 1, \ldots, m,
\label{homogeneous1LSDFP} \\
 & \sum_{i=1}^{m} y_i A_i + Y = 0,  \label{homogeneous2LSDFP} \\
 & \kappa = b^T y,  \label{homogeneous3LSDFP} \\
 & X \in S^n_{+}, y \in \Re^m, Y \in S^n_{+}, \tau \geq 0, \kappa \geq 0.
\label{homogeneous4LSDFP}
\end{eqnarray}
\end{remark}
}}

\section{Homogeneous Feasibility Model as a Semi-definite Linear Complementarity Problem}\label{sec:SDLCP}

Recall that a semi-definite linear complementarity problem (SDLCP), introduced in  \cite{Kojima1}, is given by:
\begin{eqnarray}
& {\textcolor{black}{\hat{\mathcal{A}}}}({\textcolor{black}{\hat{X}}}) + {\textcolor{black}{\hat{\mathcal{B}}}}({\textcolor{black}{\hat{Y}}})  =  q, \label{SDLCP1} \\
& {\textcolor{black}{\hat{X}}}{\textcolor{black}{\hat{Y}}}  =  0, \label{SDLCP2} \\
& {\textcolor{black}{\hat{X}}}, {\textcolor{black}{\hat{Y}}}  \in  S^{n_1}_+, \label{SDLCP3}
\end{eqnarray}
where ${\textcolor{black}{\hat{\mathcal{A}}}}, {\textcolor{black}{\hat{\mathcal{B}}}} : S^{n_1} \rightarrow \Re^{\tilde{n}_1}$ are linear operators, $q \in \Re^{\tilde{n}_1}$, and $\tilde{n}_1 = n_1(n_1+1)/2$.

We now proceed to express the homogeneous feasibility model (\ref{homogeneous1LSDFP})-(\ref{homogeneous4LSDFP}) as an SDLCP.  

%

We can write (\ref{homogeneous1LSDFP}) as
\begin{eqnarray}\label{homogeneous1prime}
\left[ \begin{array}{cc}
        {\mbox{svec}}(A_1)^T & -b_1 \\
        \vdots               & \vdots \\
        {\mbox{svec}}(A_m)^T & -b_m
        \end{array} \right]
\left[ \begin{array}{c}
        {\mbox{svec}}(X) \\
        \tau
        \end{array} \right] = 0.
\end{eqnarray}

On the other hand, combining (\ref{homogeneous2LSDFP}) and (\ref{homogeneous3LSDFP}) into one equation, we obtain
\begin{eqnarray}\label{homogeneous23prime}
\left[ \begin{array}{ccc}
       -b_1               & \ldots & -b_m \\
        {\mbox{svec}}(A_1) & \ldots & {\mbox{svec}}(A_m) 
        \end{array} \right]y +
\left[ \begin{array}{c}
		\kappa \\
            {\mbox{svec}}(Y) 
        \end{array} \right] = 0.
\end{eqnarray}

Let us now rewrite (\ref{homogeneous1prime}) and (\ref{homogeneous23prime}) in a more compact form.  Let 
\begin{eqnarray*}
\mathcal{A} := \left[ \begin{array}{c}
                        {\mbox{svec}}(A_1)^T \\
                        \vdots \\
                        {\mbox{svec}}(A_m)^T
                    \end{array} \right] \in \Re^{m \times \tilde{n}},
\end{eqnarray*}
and recall that $b = (b_1, \ldots, b_m)^T$.  Then, (\ref{homogeneous1prime}) and (\ref{homogeneous23prime}) can be written as
\begin{eqnarray}\label{eq:homogeneous1prime}
[\mathcal{A}\ - b] 
\left[ \begin{array}{c}
{\rm{svec}}(X) \\
\tau
\end{array} \right] = 0
\end{eqnarray}
and
\begin{eqnarray}\label{eq:homogeneous23prime}
\left[ \begin{array}{cc}
		- b^T  \\
		\mathcal{A}^T 		
		\end{array} \right] y + 
\left[ \begin{array}{c}
		  \kappa \\
            {\mbox{svec}}(Y)           
        \end{array} \right] = 0
\end{eqnarray}
respectively.

Let {\textcolor{black}{$B_1 \in S^n $ and  $d_1 \in \Re, d_1 \not= 0$, such that
\begin{eqnarray*}
\left[ \begin{array}{c}
		d_1 \\
		{\mbox{svec}}(B_1) 
		\end{array} \right]
\perp
\left\{ \left[ \begin{array}{c}
            -b_1 \\
            {\mbox{svec}}(A_1) 
        \end{array} \right], \ldots,
\left[ \begin{array}{c}
            -b_m \\
            {\mbox{svec}}(A_m) 
        \end{array} \right]  \right\}.
\end{eqnarray*}
}}
Since $d_1 \not= 0$ and $b_i \not= 0 $ for some $i = 1, \ldots, m$, $B_1$ is necessarily nonzero\footnote{{\textcolor{black}{We further require $B_1$ to satisfy a more restrictive condition as stated in the appendix.}}}.

Furthermore, let the following set of linearly independent vectors in $\Re^{\tilde{n} }$
\begin{eqnarray*}
\left\{  {\mbox{svec}}(B_2), \ldots,  {\mbox{svec}}(B_{\tilde{n}+1 - m}) \right\}
\end{eqnarray*}
spans the orthogonal subspace to the space spanned by
\begin{eqnarray*}
\left\{  {\mbox{svec}}(A_1), \ldots,  {\mbox{svec}}(A_m) \right\} ,
\end{eqnarray*}
\noindent where $\tilde{n} = n(n+1)/2$.  

\begin{remark}\label{rem:linearlyindependentvectors}
We have
\begin{eqnarray*}
\left\{ \left[ \begin{array}{c}
            -b_1 \\
            {\mbox{svec}}(A_1) 
        \end{array} \right], \ldots,
\left[ \begin{array}{c}
            -b_m \\
            {\mbox{svec}}(A_m) 
        \end{array} \right],
 \left[ \begin{array}{c}
 		d_1 \\
 		{\mbox{svec}}(B_1)
 		\end{array} \right],
\left[ \begin{array}{c}
		0 \\
		{\mbox{svec}}(B_2)
		\end{array} \right], \ldots,
\left[ \begin{array}{c}
		0 \\
		{\mbox{svec}}(B_{\tilde{n}+1-m}) 
		\end{array} \right] \right\}
\end{eqnarray*}
spans $\Re^{\tilde{n}+1}$, with elements in 
\begin{eqnarray*}
\left\{ \left[ \begin{array}{c}
            -b_1 \\
            {\mbox{svec}}(A_1) 
        \end{array} \right], \ldots,
\left[ \begin{array}{c}
            -b_m \\
            {\mbox{svec}}(A_m) 
        \end{array} \right] \right\}
\end{eqnarray*}
orthogonal to elements in
\begin{eqnarray*}
\left\{  \left[ \begin{array}{c}
 		d_1 \\
 		{\mbox{svec}}(B_1)
 		\end{array} \right],
\left[ \begin{array}{c}
		0 \\
		{\mbox{svec}}(B_2)
		\end{array} \right], \ldots,
\left[ \begin{array}{c}
		0 \\
		{\mbox{svec}}(B_{\tilde{n}+1-m}) 
		\end{array} \right] \right\}.
\end{eqnarray*}
\end{remark}
Let
\begin{eqnarray}\label{def:B}
\mathcal{B} := \left[ \begin{array}{c}
                        {\mbox{svec}}(B_1)^T \\
                        {\mbox{svec}}(B_2)^T \\
                        \vdots \\
                        {\mbox{svec}}(B_{\tilde{n}+1 - m})^T
                        \end{array} \right] \in \Re^{(\tilde{n}+1-m)\times \tilde{n}},\ \
d := \left[ \begin{array}{c}
                d_1 \\
                0 \\
                \vdots \\
               0
            \end{array} \right] \in \Re^{\tilde{n} + 1 - m}.
\end{eqnarray}

Then, (\ref{eq:homogeneous23prime}) holds if and only if
\begin{eqnarray}\label{eq:homogeneous23primeprime}
[d\ \ \mathcal{B}] \left[ \begin{array}{c}
\kappa \\
{\rm{svec}}(Y) 
\end{array} \right] = 0
\end{eqnarray}
since
\begin{eqnarray*}
[d\ \ \mathcal{B}]
 \left[ \begin{array}{cc}
		- b^T  \\
		\mathcal{A}^T 		
		\end{array} \right] = 0.
\end{eqnarray*}

\begin{remark}\label{rem:homogeneousmodel}
Note that for $(y,Y)$ feasible to the {\textcolor{black}{sDSDP}}, we have the entries on the left-hand side of (\ref{eq:homogeneous23primeprime}) are equal to zero, except possibly the first entry.  This follows from the way $d$ is constructed, and that for such $(y,Y)$, the left-hand side of (\ref{homogeneous23prime}) (or (\ref{eq:homogeneous23prime})) has all its entries equal to zero, except possibly the first entry.  {\textcolor{black}{By choosing $B_1$ appropriately\footnote{{\textcolor{black}{Details are provided in {\textcolor{black}{the appendix}}.}}}}}, we can then apply a result in the literature, namely, Theorem 5.1 in \cite{Sim1}, to show superlinear convergence of the interior point algorithm considered in this paper on the homogeneous feasibility model of the {\textcolor{black}{sLSDFP}}.  {\textcolor{black}{This is possible when the initial iterate to the algorithm is {\textcolor{black}{chosen}} from a strictly feasible solution to the {\textcolor{black}{sDSDP}}.}}
\end{remark}
The above development implies that (\ref{homogeneous1LSDFP})-(\ref{homogeneous3LSDFP}) can be rewritten
as
{\textcolor{black}{
\begin{eqnarray}\label{homogeneous1to3combinedprime}
\left[ \begin{array}{cccc}
        \mathcal{A} & -b & 0 & 0 \\
        0 & 0 &
        \mathcal{B}  & d
        \end{array} \right]
\left[ \begin{array}{c}
        {\mbox{svec}}(X) \\
        \tau \\
        {\mbox{svec}}(Y)  \\
        \kappa         
        \end{array} \right] = 0.
\end{eqnarray}
}}
%

We have (\ref{homogeneous1to3combinedprime}) together with
\begin{eqnarray}\label{homogeneous4prime}
X \in S^n_+, Y \in S^n_+, \tau \geq 0, \kappa \geq 0,
\end{eqnarray}
is {\textcolor{black}{another way the homogeneous feasibility model of the {\textcolor{black}{sLSDFP}} can be represented.}}

It is easy to convince ourselves that (\ref{homogeneous1to3combinedprime}), (\ref{homogeneous4prime}) is also a {\textcolor{black}{representation}} of the homogeneous feasibility model of the {\textcolor{black}{sLSDFP}}, just like (\ref{homogeneous1LSDFP})-(\ref{homogeneous4LSDFP}).

We are now ready to express the homogeneous feasibility model of the {\textcolor{black}{sLSDFP}} as SDLCP (\ref{SDLCP1})-(\ref{SDLCP3}) by letting $n_1 = n+1, q = 0$, and ${\textcolor{black}{\hat{\mathcal{A}}}}, {\textcolor{black}{\hat{\mathcal{B}}}}$ such that
\begin{eqnarray}\label{def:A1}
\begin{array}{lll}
({\textcolor{black}{\hat{\mathcal{A}}}}({\textcolor{black}{\hat{X}}}))_i & := & {\rm{Tr}}\left( \left[ \begin{array}{cc}
														A_i & 0 \\
														0 & -b_i 
													\end{array} \right] {\textcolor{black}{\hat{X}}} \right),\ \ i = 1, \ldots, m, \\
({\textcolor{black}{\hat{\mathcal{A}}}}({\textcolor{black}{\hat{X}}}))_i & := & {\rm{Tr}}(E^{i-m,n+1} {\textcolor{black}{\hat{X}}}), \ \ i = m+1, \ldots, m+n, \\
({\textcolor{black}{\hat{\mathcal{A}}}}({\textcolor{black}{\hat{X}}}))_i & := & 0, \ \ i = m + n + 1, \ldots, \tilde{n}_1,
\end{array}
\end{eqnarray}
and
\begin{eqnarray}\label{def:B1}
\begin{array}{lll}
({\textcolor{black}{\hat{\mathcal{B}}}}({\textcolor{black}{\hat{Y}}}))_j & := & 0, \ \ j = 1, \ldots, m + n, \\
({\textcolor{black}{\hat{\mathcal{B}}}}({\textcolor{black}{\hat{Y}}}))_j & := & {\rm{Tr}}\left(\left[ \begin{array}{cc}
										B_{j-(m+n)}   & 0 \\
										0 & d_{j-(m+n)}
										\end{array} \right]{\textcolor{black}{\hat{Y}}} \right), \ \ j = m + n + 1, \ldots, \tilde{n}_1,
\end{array}
\end{eqnarray}
for ${\textcolor{black}{\hat{X}}}, {\textcolor{black}{\hat{Y}}} \in S^{n_1}$.

{\textcolor{black}{We call this SDLCP, the sSDLCP from now onwards.}}

\begin{remark}\label{rem:LSDFP}
Note that the above SDLCP representation of the homogeneous feasibility model has the structure of an LSDFP.  Recall that an LSDFP written as the SDLCP (\ref{SDLCP1})-(\ref{SDLCP3}) is such that in (\ref{SDLCP1}), $({\textcolor{black}{\hat{\mathcal{A}}}}({\textcolor{black}{\hat{X}}}))_i = 0$ for $i = m_1 + 1, \ldots, \tilde{n}_1$, $({\textcolor{black}{\hat{\mathcal{B}}}}({\textcolor{black}{\hat{Y}}}))_j = 0$ for $j = 1, \ldots, m_1$, and $q_i = 0$ for $i = 1, \ldots, m_1$ or $q_i = 0$ for $i = m_1 + 1, \ldots, \tilde{n}_1$.  Here, $q = (q_1, \ldots, q_{\tilde{n}_1})^T$.  From (\ref{def:A1}), (\ref{def:B1}), we see that for our SDLCP representation, it has this structure with $m_1 = m + n$ and $q = 0$.    Having $q = 0$, in particular, $q_i = 0$ for $i = 1, \ldots, m_1$, is the property satisfied by our SDLCP representation, {\textcolor{black}{sSDLCP}}, that does not hold for the {\textcolor{black}{sLSDFP}}. {\textcolor{black}{ In fact, having this property is the key that allows us to apply results in the literature, namely \cite{Sim1}, to show the main result (Theorem \ref{thm:superlinearconvergence}) in this paper  on superlinear convergence.}} 
\end{remark}
The following proposition relates a solution of the homogeneous feasibility model of the {\textcolor{black}{sLSDFP}} to that of its SDLCP representation, {\textcolor{black}{sSDLCP}}:
\begin{proposition}\label{prop:equivalence}
We have if $(X,Y,\tau,\kappa)$ satisfies (\ref{homogeneous1to3combinedprime}), (\ref{homogeneous4prime}), then
\begin{eqnarray}\label{eq:X1Y1}
{\textcolor{black}{\hat{X}}}  =  \left[ \begin{array}{cc}
					X & 0 \\
					0 & \tau
				\end{array} \right], \quad
{\textcolor{black}{\hat{Y}}} = \left[ \begin{array}{cc}
					Y & \ast \\
					\ast & \kappa
			\end{array} \right],
\end{eqnarray}
where the ``$\ast$" entries in ${\textcolor{black}{\hat{Y}}}$ are such that ${\textcolor{black}{\hat{Y}}} \in S^{n_1}_+$, satisfies {\textcolor{black}{sSDLCP}}.  On the other hand, if $({\textcolor{black}{\hat{X}}}, {\textcolor{black}{\hat{Y}}})$ satisfies {\textcolor{black}{sSDLCP}}, then ${\textcolor{black}{\hat{X}}}, {\textcolor{black}{\hat{Y}}}$ are given by (\ref{eq:X1Y1}), and $(X, Y, \tau, \kappa)$  satisfies (\ref{homogeneous1to3combinedprime}), (\ref{homogeneous4prime}).
\end{proposition}
\begin{proof}
The proposition is clear based on how ${\textcolor{black}{\hat{\mathcal{A}}}}$ and ${\textcolor{black}{\hat{\mathcal{B}}}}$ are defined in (\ref{def:A1}) and (\ref{def:B1}) respectively.
\end{proof}

\vspace{10pt}

\noindent {\textcolor{black}{The following assumptions are imposed on the SDLCP (\ref{SDLCP1})-(\ref{SDLCP3}), which we show in Proposition \ref{prop:SDLCPAssumption} to hold for the SDLCP representation, sSDLCP, of the homogeneous feasibility model of the sLSDFP.
\begin{assumption}\label{ass:SDLCPassumptions}
\begin{enumerate}[(a)]
\item System (\ref{SDLCP1})-(\ref{SDLCP3}) is monotone.  That is, ${\textcolor{black}{\hat{\mathcal{A}}}}({\textcolor{black}{\hat{X}}}) + {\textcolor{black}{\hat{\mathcal{B}}}}({\textcolor{black}{\hat{Y}}}) = 0$ for ${\textcolor{black}{\hat{X}}}, {\textcolor{black}{\hat{Y}}} \in S^{n_1} \Rightarrow {\rm{Tr}}({\textcolor{black}{\hat{X}}}{\textcolor{black}{\hat{Y}}}) \geq 0$.
\item There exists at least one solution to SDLCP (\ref{SDLCP1})-(\ref{SDLCP3}).
\item $\{ {\textcolor{black}{\hat{\mathcal{A}}}}({\textcolor{black}{\hat{X}}}) + {\textcolor{black}{\hat{\mathcal{B}}}}({\textcolor{black}{\hat{Y}}})\ ; \ {\textcolor{black}{\hat{X}}}, {\textcolor{black}{\hat{Y}}} \in S^{n_1} \} = \Re^{\tilde{n}_1}$.
\end{enumerate}
\end{assumption}
}}

\noindent {\textcolor{black}{When the SDLCP (\ref{SDLCP1})-(\ref{SDLCP3}) is studied in some papers in the literature, instead of Assumption \ref{ass:SDLCPassumptions}(b), the following assumption is imposed:
\begin{assumption}\label{ass:strictfeasibility}
There exist ${\textcolor{black}{\hat{X}}}, {\textcolor{black}{\hat{Y}}} \in S^{n_1}_{++}$ such that ${\textcolor{black}{\hat{\mathcal{A}}}}({\textcolor{black}{\hat{X}}}) + {\textcolor{black}{\hat{\mathcal{B}}}}({\textcolor{black}{\hat{Y}}}) = q$.
\end{assumption}
Assumptions \ref{ass:SDLCPassumptions}(a), (c) and \ref{ass:strictfeasibility} are imposed in \cite{Kojima1} where the paper studies feasible interior point algorithms on the SDLCP (\ref{SDLCP1})-(\ref{SDLCP3}), while \cite{Kojima,Sim1} assume Assumption \ref{ass:SDLCPassumptions} in the study of infeasible interior point algorithms on the SDLCP (\ref{SDLCP1})-(\ref{SDLCP3}).  In the study of infeasible interior point algorithms on the SDLCP (\ref{SDLCP1})-(\ref{SDLCP3}), it is not necessary to impose Assumption \ref{ass:strictfeasibility}.  A reason {{is}} that we do not need a strictly feasible initial iterate to the algorithm.  To solve the primal-dual SDP pair (\ref{primalSDP})-(\ref{dualSDP}) by applying an interior point algorithm on its homogeneous feasibility model (\ref{homogeneous1})-(\ref{homogeneous4}), the algorithm is necessarily infeasible as discussed near the end of Section \ref{sec:LSDFPhomogeneousmodel}.  When we express the homogeneous feasibility model as an SDLCP, the interior point algorithm that is applied to the homogeneous feasibility model can be equivalently applied to the corresponding SDLCP, and by necessity, the algorithm on the SDLCP is infeasible just like that for the homogeneous feasibility model.  Therefore, having Assumption \ref{ass:strictfeasibility} imposed on the SDLCP is not suitable and in fact can never hold in our case as there cannot exist an $(X, Y) \in S^n_{++} \times S^n_{++}$ such that (\ref{SDLCP1}) holds for the SDLCP obtained from the homogeneous feasibility model.  We therefore have Assumption \ref{ass:SDLCPassumptions}(b) in its place.  In the context when the primal-dual SDP pair (\ref{primalSDP})-(\ref{dualSDP}) has $C = 0$, we see in Proposition \ref{prop:SDLCPAssumption} below that when the sLSDFP satisfies Assumption \ref{ass:SDP}, {{the SDLCP representation, sSDLCP, of the homogeneous feasibility model (\ref{homogeneous1LSDFP})-(\ref{homogeneous4LSDFP}) satisfies Assumption \ref{ass:SDLCPassumptions}.}}}}

\begin{proposition}\label{prop:SDLCPAssumption}
{\textcolor{black}{We have the sSDLCP}} satisfies Assumption \ref{ass:SDLCPassumptions} when the {\textcolor{black}{the sLSDFP}} satisfies Assumption \ref{ass:SDP}.
\end{proposition}
\begin{proof}
We first make two observations.  Firstly, if $(X, Y, \tau, \kappa)$ satisfies  (\ref{homogeneous1to3combinedprime}), then ${\rm{Tr}}(XY) + \tau \kappa = 0$, by Remark \ref{rem:linearlyindependentvectors}.  Secondly, the matrix on the left-hand side of (\ref{homogeneous1to3combinedprime}) has full row rank by Assumption \ref{ass:SDP}(b), $\{ {\rm{svec}}(B_2), \ldots, {\rm{svec}}(B_{\tilde{n} + 1 - m}) \}$ linearly independent, and $[{\rm{svec}}(B_1)^T \ d_1]^T$ linearly independent from $\{ [{\rm{svec}}(B_i)^T\ 0]^T \ ; \ i = 2, \ldots, \tilde{n} + 1 - m \}$.

\noindent Given the {\textcolor{black}{sSDLCP}}.  Suppose
\begin{eqnarray*}
{\textcolor{black}{\hat{\mathcal{A}}}}({\textcolor{black}{\hat{X}}}) + {\textcolor{black}{\hat{\mathcal{B}}}}({\textcolor{black}{\hat{Y}}}) = 0
\end{eqnarray*}
for some $({\textcolor{black}{\hat{X}}}, {\textcolor{black}{\hat{Y}}}) \in S^{n_1} \times S^{n_1}$.  Then ${\textcolor{black}{\hat{X}}}$ and ${\textcolor{black}{\hat{Y}}}$ are given  by
\begin{eqnarray*}
{\textcolor{black}{\hat{X}}}  =  \left[ \begin{array}{cc}
					X & 0 \\
					0 & \tau
				\end{array} \right], \quad
{\textcolor{black}{\hat{Y}}} = \left[ \begin{array}{cc}
					Y & \ast \\
					\ast & \kappa
			\end{array} \right],
\end{eqnarray*} 
where $(X,Y, \tau, \kappa)$ satisfies (\ref{homogeneous1to3combinedprime}).  Hence, by the first observation above, we have ${\rm{Tr}}({\textcolor{black}{\hat{X}}} {\textcolor{black}{\hat{Y}}}) = {\rm{Tr}}(XY) + \tau \kappa = 0$.  Therefore, Assumption \ref{ass:SDLCPassumptions}(a) holds, {\textcolor{black}{since $\hat{A}(\hat{X}) + \hat{B}(\hat{Y}) = 0$ for $\hat{X}, \hat{Y} \in S^{n_1}$ implies that ${\mbox{Tr}}(\hat{X}\hat{Y}) = 0$, as argued above}}.  Furthermore, Assumption \ref{ass:SDLCPassumptions}(b) holds since a solution to the given {\textcolor{black}{sSDLCP}} is ${\textcolor{black}{\hat{X}}} = 0, {\textcolor{black}{\hat{Y}}} = 0$.  Finally, the second observation above means that the matrix
\begin{eqnarray*}
\left[ \begin{array}{cccc}
        \mathcal{A} & -b & 0 & 0 \\
        0 & 0 &
        \mathcal{B} & d
        \end{array} \right]
\end{eqnarray*}
has full row rank.  This implies that the matrix $({\textcolor{black}{\hat{\mathcal{A}}}} \ {\textcolor{black}{\hat{\mathcal{B}}}})$, where ${\textcolor{black}{\hat{\mathcal{A}}}}$ and ${\textcolor{black}{\hat{\mathcal{B}}}}$ are defined by (\ref{def:A1}) and (\ref{def:B1}) respectively, has full row rank, and hence Assumption \ref{ass:SDLCPassumptions}(c) holds as well. 
\end{proof}

\section{An Infeasible Interior Point
Algorithm on the Homogeneous Feasibility Model}\label{sec:interiorpointalgorithm}

We describe in this section an infeasible path-following interior point algorithm, {\textcolor{black}{Algorithm \ref{mainalgorithm}}}, on the homogeneous
feasibility model (\ref{homogeneous1LSDFP})-(\ref{homogeneous4LSDFP}) (or (\ref{homogeneous1to3combinedprime}), (\ref{homogeneous4prime})).  It generates iterates that follow an infeasible
central path in a (narrow) neighborhood.  This algorithm is a predictor-corrector type algorithm on the homogeneous feasibility model, and is first considered in \cite{Potra}.

Consider the following (narrow) neighborhood of the central path:
\begin{eqnarray*}
\mathcal{N}(\beta,\mu) & := & \{(X, y, Y,\tau,\kappa) \in S^n_{++}
\times \Re^m \times S^n_{++} \times \Re_{++} \times \Re_{++}\ ;\\
& & \ \ \ \ (\|Y^{1/2}XY^{1/2} - \mu I \|_F^2 + (\tau \kappa - \mu)^2)^{1/2}
\leq \beta \mu,\ \mu = ({\rm{Tr}}(XY) + \tau \kappa)/(n+1) \}.
\end{eqnarray*}

{\textcolor{black}{In Algorithm \ref{mainalgorithm}, which is described below, iterates}} always stay
within a neighborhood of the central path.  {\textcolor{black}{We consider the dual Helmberg-Kojima-Monteiro (HKM) search direction in the algorithm, even though our results also hold for the Nesterov-Todd (NT) search direction - see Remark \ref{rem:NT}.}}  Among different search directions used in interior point algorithms on SDLCPs/SDPs, the Alizadeh-Haeberly-Overton (AHO) \cite{Alizadeh}, Helmberg-Kojima-Monteiro (HKM) \cite{Helmberg,Kojima1,Monteiro3} and Nesterov-Todd (NT) \cite{Nesterov,Nesterov1} search directions are better known, with the latter two being implemented in SDP solvers, such as SeDuMi \cite{Sturm1} and SDPT3 \cite{Toh1}.  {\textcolor{black}{It is worth noting that the HKM direction has also been implemented in the SDPA package to solve SDPs \cite{Yamashita,Yamashita2}.}}
  
{\textcolor{black}{Before describing the algorithm, let us consider}} the following system of equations
for $(\Delta X, \Delta y, \Delta Y,\Delta \tau, \Delta \kappa) \in S^n \times \Re^m \times S^n \times \Re \times \Re$ which plays
an important role in the algorithm:
\begin{eqnarray}
 & Y^{1/2}(X \Delta Y +  \Delta X Y)Y^{-1/2} + Y^{-1/2}(\Delta Y X + Y \Delta X)Y^{1/2}  =  2(\sigma
\mu I
- Y^{1/2}XY^{1/2}), \label{sys1}\\
 & \kappa \Delta \tau + \tau \Delta \kappa =  \sigma \mu - \tau
\kappa, \label{sys2}\\
 & {\rm{Tr}}(A_i \Delta X) - b_i \Delta \tau  =  -\overline{r}_i, \quad i = 1, \ldots, m,\label{sys3} \\
 & \sum_{i=1}^m \Delta y_i A_i + \Delta Y  =  -
\overline{s}, \label{sys4} \\
 & \Delta \kappa - b^T \Delta y  =  - \overline{\gamma}, \label{sys5}
\end{eqnarray}
where $\mu = ({\rm{Tr}}(XY) + \tau \kappa)/(n+1)$.  {\textcolor{black}{Note that (\ref{sys1}) and (\ref{sys2}) arise upon linearizing the perturbed equations to $XY = 0$ and $\tau \kappa = 0$ when we use the dual HKM search direction. Furthemore, it can be shown that in Algorithm \ref{mainalgorithm}, $(\Delta X, \Delta y, \Delta Y, \Delta \tau, \Delta \kappa)$ in (\ref{sys1})-(\ref{sys5}) is always uniquely determined - see for example \cite{Potra}.}}

%

The infeasible predictor-corrector path-following interior point algorithm
 on the homogeneous feasibility model (\ref{homogeneous1LSDFP}) - (\ref{homogeneous4LSDFP})  is described as follows:
\begin{algorithm}\label{mainalgorithm} (See Algorithm 5.1 of \cite{Potra})
{\textcolor{black}{Given $\epsilon > 0$, and }} $\beta_1 < \beta_2$ with $\beta_2^2/(2(1-\beta_2)^2) \leq
\beta_1 < \beta_2 < \beta_2/(1-\beta_2) < 1$.  Choose $(X_0,y_0,Y_0,
\tau_0,\kappa_0) \in \mathcal{N}(\beta_1, \mu_0)$ with $(n+1) \mu_0
= {\rm{Tr}}(X_0 Y_0) + \tau_0 \kappa_0$. For $k = 0, 1, \ldots$,
{\textcolor{black}{perform}} ($a1$) through ($a5$):
\begin{description}
\item[\hspace{5pt} ($a1$)] Set $X = X_k$, $y = y_k$, $Y = Y_k$, $\tau =
\tau_k$, $\kappa = \kappa_k$, and define
\begin{eqnarray*}
r_i & := & {\rm{Tr}}(A_i X) - b_i \tau, \quad i = 1, \ldots, m, \\
s & := & \sum_{i=1}^m y_i A_i + Y, \\
\gamma & := & \kappa - b^T y.
\end{eqnarray*}
\item[\hspace{5pt} ($a2$)] If $\max \{({\rm{Tr}}(XY) + \tau
\kappa)/\tau^2 , | r_1/\tau |, \ldots, | r_m /\tau |, \| s/\tau \|  \} \leq \epsilon$, then
report $(X/\tau, y/\tau, Y/\tau)$ as an approximate solution to the {\textcolor{black}{sLSDFP}}, and terminate.
If $\tau$ is sufficiently small, terminate with no optimal solutions
to the {\textcolor{black}{sLSDFP}}.
\item[\hspace{5pt} ($a3$)] {\bf{[Predictor Step]}} Find the solution $(\Delta X_p,\Delta y_p, \Delta Y_p,\Delta \tau_p,
\Delta \kappa_p)$ of the linear system (\ref{sys1})-(\ref{sys5}), with
$\sigma = 0$, $\overline{r}_i = r_i, i = 1, \ldots, m$, $\overline{s} = s$ and $\overline{\gamma} = \gamma$. \newline
Define
\begin{eqnarray*}
\overline{X} = X + \overline{\alpha} \Delta X_p,\ \ \overline{y} = y + \overline{\alpha} \Delta y_p, \ \ \overline{Y} = Y +
\overline{\alpha} \Delta Y_p,\ \  \overline{\tau} = \tau +
\overline{\alpha}\Delta \tau_p,\ \ \overline{\kappa} = \kappa +
\overline{\alpha}\Delta \kappa_p,
\end{eqnarray*}
\noindent where the {\textcolor{black}{step length}} $\overline{\alpha}$ satisfies
\begin{eqnarray}\label{steplengthinequality}
\alpha_1 \leq \overline{\alpha} \leq \alpha_2.
\end{eqnarray}
Here,
\begin{eqnarray}
\alpha_1 & = & \frac{2}{\sqrt{1 + 4 \delta/(\beta_2 - \beta_1)} +1},
\label{steplengthinequality1}\\
\delta & = & \frac{1}{\mu} \left\| \left[ \begin{array}{cc}
                                                Y & 0 \\
                                                0 & \kappa
                                            \end{array}\right]^{1/2}
\left[ \begin{array}{cc}
            \Delta X_p & 0  \\
            0 & \Delta \tau_p
        \end{array} \right]
\left[ \begin{array}{cc}
          \Delta Y_p   & 0 \\
            0 & \Delta \kappa_p 
        \end{array} \right]
\left[ \begin{array}{cc}
           Y    & 0 \\
            0 & \kappa 
        \end{array} \right]^{-1/2} \right\|_F,
        \label{steplengthinequality2}
\end{eqnarray}
where 
\begin{eqnarray*}
\mu = \frac{{\rm{Tr}}(XY) + \tau \kappa}{n + 1},
\end{eqnarray*}
and
\begin{eqnarray*}
& & \alpha_2 = \max \{ \tilde{\alpha} \in [0,1] \ ; \ (X+ \alpha \Delta X_p, y + \alpha \Delta y_p,  Y +
\alpha \Delta Y_p, \tau + \alpha \Delta \tau_p, \kappa + \alpha \Delta \kappa_p)
\\ 
& & \ \ \ \ \ \ \ \ \ \ \ \ \ \ \ \ \ \ \ \ \ \ \ \ \ \ \ \ \ \ \ \ \in \mathcal{N}(\beta_2,(1-\alpha)\mu)\ \forall\ \alpha \in
[0,\tilde{\alpha}]\}.
\end{eqnarray*}
\item[\hspace{5pt} ($a4$)] {\bf{[Corrector Step]}} Find the solution $(\Delta X_c, \Delta y_c, \Delta Y_c, \Delta \tau_c, \Delta \kappa_c)$ of the
linear system (\ref{sys1})-(\ref{sys5}), with ${\textcolor{black}{X = \overline{X}, y = \overline{y}, Y = \overline{Y}, \tau = \overline{\tau}, \kappa = \overline{\kappa}}}, \sigma = 1 -
\overline{\alpha}$, $\overline{r}_i = 0, i = 1, \ldots, m$, $\overline{s} = 0$ and $\overline{\gamma} = 0$.  Set
\begin{eqnarray*}
\begin{array}{c}
X_+ = \overline{X} + \Delta X_c,\ \ y_+ = \overline{y} + \Delta y_c,\ \ Y_+ = \overline{Y} + \Delta Y_c, \ \ \tau_+ =
\overline{\tau} + \Delta \tau_c,\ \ \kappa_+ = \overline{\kappa} +
\Delta \kappa_c, \\
\mu_+ = (1 - \overline{\alpha})\mu.
\end{array}
\end{eqnarray*}
\item[\hspace{5pt} ($a5$)] Set
\begin{eqnarray*}
\begin{array}{c}
X_{k+1} = X_+,\ \ y_{k+1} = y_+, \ \ Y_{k+1} = Y_+, \ \ \tau_{k+1} = \tau_+, \ \
\kappa_{k+1} = \kappa_+,\\
\mu_{k+1} = \mu_+.
\end{array}
\end{eqnarray*}
\end{description}
\end{algorithm}

\noindent Polynomial iteration complexity for a general version of Algorithm \ref{mainalgorithm} used to solve (\ref{homogeneous1})-(\ref{homogeneous4}) is shown in \cite{Potra} where the HKM direction is considered.  The result there also holds for the dual HKM direction.  Hence, Algorithm \ref{mainalgorithm} has polynomial iteration complexity when it is used to solve (\ref{homogeneous1LSDFP})-(\ref{homogeneous4LSDFP}).   Furthermore, we remark that Algorithm \ref{mainalgorithm} can easily be adapted to solve the {\textcolor{black}{sLSDFP}} instead of its homogeneous feasibility model.  An advantage of applying the algorithm {\textcolor{black}{to}} its homogeneous feasibility model is that we have guaranteed superlinear convergence of iterates generated by the algorithm, as shown in Theorem \ref{thm:superlinearconvergence}.


\begin{remark}\label{rem:Remark1}
For all $k \geq 0$, we have
\begin{eqnarray*}
(X_k, y_k, Y_k, \tau_k, \kappa_k) \in \mathcal{N}(\beta_1, \mu_k).
\end{eqnarray*}
\end{remark}
Remark \ref{rem:Remark1} holds by \cite{Potra} - see also \cite{Potra1,Sim1}.

\begin{remark}\label{rem:initialiterate}
Throughout this paper, we consider Algorithm \ref{mainalgorithm} with initial iterate $(X_0, y_0, Y_0,$ $\tau_0, \kappa_0) \in \mathcal{N}(\beta_1, \mu_0)$ such that $(y_0, Y_0) \in \Re^m \times S^n_{++}$ is feasible to the {\textcolor{black}{sDSDP}}.  Therefore, in the algorithm, we have {\textcolor{black}{$\sum_{i=1}^m (y_0)_iA_i + Y_0 = 0$ while ${\mbox{Tr}}(A_i X_0) - b_i \tau_0 , i =1, \ldots, m$, and $\kappa_0 - b^T y_0$}} are generally nonzero.
\end{remark}

\noindent {\textcolor{black}{Algorithm \ref{mainalgorithm} can be written in an equivalent way as Algorithm \ref{mainalgorithm2}, which we present below.  We use Algorithm \ref{mainalgorithm2} to solve the {\textcolor{black}{sSDLCP}}.}}  Before describing the algorithm,  we first define  an analogous (narrow) neighborhood of the central path of the SDLCP representation:
\begin{eqnarray*}
\mathcal{N}_1(\beta, {\textcolor{black}{\hat{\mu}}})  :=  \{ ({\textcolor{black}{\hat{X}}}, {\textcolor{black}{\hat{Y}}}) \in S^{n_1}_{++} \times S^{n_1}_{++}\ ; \ \| ({\textcolor{black}{\hat{Y}}})^{1/2} {\textcolor{black}{\hat{X}}} ({\textcolor{black}{\hat{Y}}})^{1/2} - {\textcolor{black}{\hat{\mu}}} I \|_F \leq \beta {\textcolor{black}{\hat{\mu}}},  {\textcolor{black}{\hat{\mu}}} = {\rm{Tr}}({\textcolor{black}{\hat{X}}} {\textcolor{black}{\hat{Y}}})/n_1 \}.
\end{eqnarray*}
{\textcolor{black}{We again have a system of equations for $(\Delta {\textcolor{black}{\hat{X}}}, \Delta {\textcolor{black}{\hat{Y}}}) \in S^{n_1} \times S^{n_1}$ that plays an important role in the algorithm, just like the system of equations (\ref{sys1})-(\ref{sys5}) for Algorithm \ref{mainalgorithm}:}}
\begin{eqnarray}
&  ({\textcolor{black}{\hat{Y}}})^{1/2}({\textcolor{black}{\hat{X}}} \Delta {\textcolor{black}{\hat{Y}}} + \Delta {\textcolor{black}{\hat{X}}} {\textcolor{black}{\hat{Y}}})({\textcolor{black}{\hat{Y}}})^{-1/2} + ({\textcolor{black}{\hat{Y}}})^{-1/2}(\Delta {\textcolor{black}{\hat{Y}}} {\textcolor{black}{\hat{X}}} + {\textcolor{black}{\hat{Y}}} \Delta {\textcolor{black}{\hat{X}}})({\textcolor{black}{\hat{Y}}})^{1/2} \nonumber \\ 
&   = 2(\sigma {\textcolor{black}{\hat{\mu}}} I - ({\textcolor{black}{\hat{Y}}})^{1/2} {\textcolor{black}{\hat{X}}} ({\textcolor{black}{\hat{Y}}})^{1/2}), \label{eq:sys1prime} \\
&  {\textcolor{black}{\hat{\mathcal{A}}}}(\Delta {\textcolor{black}{\hat{X}}}) + {\textcolor{black}{\hat{\mathcal{B}}}}(\Delta {\textcolor{black}{\hat{Y}}}) = - \overline{r}. \label{eq:sys2prime}
\end{eqnarray}
{\textcolor{black}{Note that in Algorithm \ref{mainalgorithm2}, it can be shown that $(\Delta {\textcolor{black}{\hat{X}}}, \Delta {\textcolor{black}{\hat{Y}}})$ obtained by solving (\ref{eq:sys1prime}), (\ref{eq:sys2prime}) always exists and is unique \cite{Potra1, Sim1}.}}

{\textcolor{black}{Below, we describe the algorithm:}}
\begin{algorithm}\label{mainalgorithm2} (See Algorithm 4.1 of \cite{Sim1}; Algorithm 2.1 of \cite{Potra1})
{\textcolor{black}{Given $\epsilon > 0$, and}} $\beta_1 < \beta_2$ with $\beta_2^2/(2(1-\beta_2)^2) \leq
\beta_1 < \beta_2 < \beta_2/(1-\beta_2) < 1$.  Choose $({\textcolor{black}{\hat{X}}}_0,{\textcolor{black}{\hat{Y}}}_0) \in S^{n_1}_{++} \times S^{n_1}_{++}$ such that 
\begin{eqnarray}\label{def:initialiterateSDLCP}
{\textcolor{black}{\hat{X}}}_0 = \left[ \begin{array}{cc}
					X_0 & 0 \\
					0   & \tau_0
				\end{array} \right], \quad
{\textcolor{black}{\hat{Y}}}_0 = \left[ \begin{array}{cc}
					Y_0 & 0 \\
					0 & \kappa_0
				\end{array} \right],
\end{eqnarray}
where $X_0,Y_0,\tau_0,\kappa_0$ are from the initial iterate in Algorithm \ref{mainalgorithm}. For $k = 0, 1, \ldots$,
{\textcolor{black}{perform}} ($a1$) through ($a5$):
\begin{description}
\item[\hspace{5pt} ($a1$)] Set ${\textcolor{black}{\hat{X}}} = {\textcolor{black}{\hat{X}}}_k$, ${\textcolor{black}{\hat{Y}}} = {\textcolor{black}{\hat{Y}}}_k$, and define
\begin{eqnarray*}
r := {\textcolor{black}{\hat{\mathcal{A}}}}({\textcolor{black}{\hat{X}}}) + {\textcolor{black}{\hat{\mathcal{B}}}}({\textcolor{black}{\hat{Y}}}).
\end{eqnarray*}
\item[\hspace{5pt} ($a2$)] If $\max \{{\rm{Tr}}({\textcolor{black}{\hat{X}}}{\textcolor{black}{\hat{Y}}})/({\textcolor{black}{\hat{X}}})^2_{n_1,n_1} ,  \| r /{\textcolor{black}{\hat{X}}}_{n_1,n_1} \|  \} \leq \epsilon$, then terminate with the solution $({\textcolor{black}{\hat{X}}},{\textcolor{black}{\hat{Y}}})$.
If ${\textcolor{black}{\hat{X}}}_{n_1,n_1}$ is sufficiently small, terminate with no optimal solutions to the {\textcolor{black}{sLSDFP}}.
\item[\hspace{5pt} ($a3$)] {\bf{[Predictor Step]}} Find the solution $(\Delta {\textcolor{black}{{\hat{X}}}}_p, \Delta {\textcolor{black}{\hat{Y}}}_p)$ of the linear system (\ref{eq:sys1prime}), (\ref{eq:sys2prime}), with
$\sigma = 0$, $\overline{r} = r$. \newline
Define
\begin{eqnarray*}
{\textcolor{black}{\overline{\hat{X}}}} = {\textcolor{black}{\hat{X}}} + \overline{\alpha} \Delta {\textcolor{black}{\hat{X}}}_p,\ \ {\textcolor{black}{\overline{\hat{Y}}}} = {\textcolor{black}{\hat{Y}}} +
\overline{\alpha} \Delta {\textcolor{black}{\hat{Y}}}_p,
\end{eqnarray*}
\noindent where the {\textcolor{black}{step length}} $\overline{\alpha}$ satisfies
\begin{eqnarray}\label{steplengthinequalityprime}
\alpha_1 \leq \overline{\alpha} \leq \alpha_2.
\end{eqnarray}
Here,
\begin{eqnarray}
\alpha_1 & = & \frac{2}{\sqrt{1 + 4 \delta/(\beta_2 - \beta_1)} +1},
\label{steplengthinequality1prime}\\
\delta & = & \frac{1}{{\textcolor{black}{\hat{\mu}}}} \| ({\textcolor{black}{\hat{Y}}})^{1/2} \Delta X_p^1 \Delta Y_p^1 ({\textcolor{black}{\hat{Y}}})^{-1/2} \|_F,
        \label{steplengthinequality2prime}
\end{eqnarray}
where 
\begin{eqnarray*}
{\textcolor{black}{\hat{\mu}}} = \frac{{\rm{Tr}}({\textcolor{black}{\hat{X}}}{\textcolor{black}{\hat{Y}}})}{n_1},
\end{eqnarray*}
and
\begin{eqnarray*}
\alpha_2 = \max \{ \tilde{\alpha} \in [0,1] \ ; \ ({\textcolor{black}{\hat{X}}}+ \alpha \Delta {\textcolor{black}{\hat{X}}}_p, {\textcolor{black}{\hat{Y}}} +
\alpha \Delta {\textcolor{black}{\hat{Y}}}_p) \in \mathcal{N}_1(\beta_2,(1-\alpha){\textcolor{black}{\hat{\mu}}})\ \forall\ \alpha \in
[0,\tilde{\alpha}]\}.
\end{eqnarray*}
\item[\hspace{5pt} ($a4$)] {\bf{[Corrector Step]}} Find the solution $(\Delta {\textcolor{black}{\hat{X}}}_c, \Delta {\textcolor{black}{\hat{Y}}}_c)$ of the
linear system (\ref{eq:sys1prime}), (\ref{eq:sys2prime}), with ${\textcolor{black}{\hat{X} = \overline{\hat{X}},  \hat{Y} = \overline{\hat{Y}}}}, \sigma = 1 -
\overline{\alpha}$ and $\overline{r} = 0$.  Set
\begin{eqnarray*}
\begin{array}{c}
{\textcolor{black}{\hat{X}}}_+ = {\textcolor{black}{\overline{\hat{X}}}} + \Delta {\textcolor{black}{\hat{X}}}_c,\ \ {\textcolor{black}{\hat{Y}}}_+ = {\textcolor{black}{\overline{\hat{Y}}}} + \Delta {\textcolor{black}{\hat{Y}}}_c, \\
{\textcolor{black}{\hat{\mu}}}_+ = (1 - \overline{\alpha}){\textcolor{black}{\hat{\mu}}}.
\end{array}
\end{eqnarray*}
\item[\hspace{5pt} ($a5$)] Set
\begin{eqnarray*}
\begin{array}{c}
{\textcolor{black}{\hat{X}}}_{k+1} = {\textcolor{black}{\hat{X}}}_+,\ \ {\textcolor{black}{\hat{Y}}}_{k+1} = {\textcolor{black}{\hat{Y}}}_+, \\
{\textcolor{black}{\hat{\mu}}}_{k+1} = {\textcolor{black}{\hat{\mu}}}_+.
\end{array}
\end{eqnarray*}
\end{description}
\end{algorithm}

\noindent {\textcolor{black}{We see that Algorithm \ref{mainalgorithm2} is similar to Algorithm \ref{mainalgorithm}.  We present the two algorithms in full above as the proof of Proposition \ref{prop:iteratesbothalgorithms} requires us  to make a comparison between them.  Having details of both algorithms presented serves the purpose to improve readability of the proof of the proposition.}}

The following proposition relates the $k^{th}$ iterate in the two algorithms:
\begin{proposition}\label{prop:iteratesbothalgorithms}
For all $k \geq 0$,
\begin{eqnarray}\label{eq:kiterateSDLCP}
{\textcolor{black}{\hat{X}}}_k = \left[ \begin{array}{cc}
					X_k & 0 \\
					0   & \tau_k
				\end{array} \right], \quad
{\textcolor{black}{\hat{Y}}}_k = \left[ \begin{array}{cc}
					Y_k & 0 \\
					0 & \kappa_k
				\end{array} \right].
\end{eqnarray}
Consequently, ${\textcolor{black}{\hat{\mu}}}_k = \mu_k$ and $({\textcolor{black}{\hat{X}}}_k, {\textcolor{black}{\hat{Y}}}_k) \in \mathcal{N}_1(\beta_1, {\textcolor{black}{\hat{\mu}}}_k)$.
\end{proposition}
\begin{proof}
We show (\ref{eq:kiterateSDLCP}) holds by induction on $k \geq 0$.  It is clear that (\ref{eq:kiterateSDLCP}) holds when $k = 0$, by the choice of ${\textcolor{black}{\hat{X}}}_0, {\textcolor{black}{\hat{Y}}}_0$.  Suppose (\ref{eq:kiterateSDLCP}) holds for $k = k_0 \geq 0$.  Then, at the $(k_0 + 1)^{\rm{th}}$ iteration of Algorithm \ref{mainalgorithm2},  by comparing the system of equations (\ref{sys1})-(\ref{sys5}) and (\ref{eq:sys1prime}), (\ref{eq:sys2prime}), it can be seen easily that
\begin{eqnarray*}
\Delta {\textcolor{black}{\hat{X}}}_p = \left[ \begin{array}{cc}
					\Delta X_p & 0 \\
					0   & \Delta \tau_p
				\end{array} \right], \quad
\Delta {\textcolor{black}{\hat{Y}}}_p = \left[ \begin{array}{cc}
					\Delta Y_p & 0 \\
					0 & \Delta \kappa_p
				\end{array} \right]
\end{eqnarray*}
satisfy (\ref{eq:sys1prime}), (\ref{eq:sys2prime}) when $\sigma = 0, \overline{r} = r = {\textcolor{black}{\hat{\mathcal{A}}}}({\textcolor{black}{\hat{X}}}) + {\textcolor{black}{\hat{\mathcal{B}}}}({\textcolor{black}{\hat{Y}}})$.  Furthermore, the {\textcolor{black}{step length}} $\overline{\alpha}$ in the $(k_0 + 1)^{\rm{th}}$ iteration of both algorithms are the same.  These lead to
\begin{eqnarray}\label{eq:overlineX1Y1withoverlineXY}
{\textcolor{black}{\overline{\hat{X}}}} = \left[ \begin{array}{cc}
					\overline{X} & 0 \\
					0   & \overline{\tau}
				\end{array} \right], \quad
{\textcolor{black}{\overline{\hat{Y}}}} = \left[ \begin{array}{cc}
					\overline{Y} & 0 \\
					0 & \overline{\kappa}
				\end{array} \right]
\end{eqnarray}
{\textcolor{black}{in the two algorithms}}.  With (\ref{eq:overlineX1Y1withoverlineXY}), again comparing the system of equations (\ref{sys1})-(\ref{sys5}) and (\ref{eq:sys1prime}), (\ref{eq:sys2prime}), it is also easy to see that
\begin{eqnarray*}
\Delta {\textcolor{black}{\hat{X}}}_c = \left[ \begin{array}{cc}
					\Delta X_c & 0 \\
					0   & \Delta \tau_c
				\end{array} \right], \quad
\Delta {\textcolor{black}{\hat{Y}}}_c = \left[ \begin{array}{cc}
					\Delta Y_c & 0 \\
					0 & \Delta \kappa_c
				\end{array} \right]
\end{eqnarray*}
satisfy (\ref{eq:sys1prime}), (\ref{eq:sys2prime}) when $\sigma = 1 - \overline{\alpha}$ and $\overline{r} = 0$.  Hence, we conclude that (\ref{eq:kiterateSDLCP}) holds for $k = k_0 + 1$.  Therefore, by induction, (\ref{eq:kiterateSDLCP}) holds for all $k \geq 0$.  Furthermore, we have
\begin{eqnarray*}
{\textcolor{black}{\hat{\mu}}}_k = \frac{{\rm{Tr}}({\textcolor{black}{\hat{X}}}_k {\textcolor{black}{\hat{Y}}}_k)}{n_1} = \frac{{\rm{Tr}}(X_kY_k) + \tau_k \kappa_k}{n + 1} = \mu_k.
\end{eqnarray*}
Finally, by (\ref{eq:kiterateSDLCP}), comparing the definition of the neighborhood $\mathcal{N}(\beta, \mu)$ and the neighborhood $\mathcal{N}_1(\beta, {\textcolor{black}{\hat{\mu}}})$, and that $\mu_k = {\textcolor{black}{\hat{\mu}}}_k$, we see that since $(X_k, y_k, Y_k, \tau_k, \kappa_k) \in \mathcal{N}(\beta_1, \mu_k)$ (Remark \ref{rem:Remark1}), we have $({\textcolor{black}{\hat{X}}}_k,{\textcolor{black}{\hat{Y}}}_k) \in \mathcal{N}_1(\beta_1,{\textcolor{black}{\hat{\mu}}}_k)$. 
\end{proof}

%

\subsection{Superlinear Convergence}\label{subsec:superlinearconvergence}

We show in this subsection that Algorithm \ref{mainalgorithm} applied to the homogeneous feasibility model {\textcolor{black}{of sLSDFP}} when the initial iterate $(X_0,y_0,Y_0,\tau_0,\kappa_0)  \in \mathcal{N}(\beta_1, \mu_0)$ is such that  $(y_0,Y_0) \in \Re^m \times S^n_{++}$ is feasible to the {\textcolor{black}{sDSDP}} leads to superlinear convergence of iterates generated by the algorithm, besides polynomial iteration complexity.   
First, we state an additional assumption, Assumption \ref{ass:strictcomplementarity}, which is the assumption of strict complementarity, on the primal-dual SDP pair (\ref{primalSDP})-(\ref{dualSDP}) with $C = 0$ that is needed for this result to hold.   Note that strict complementarity assumption is generally considered the minimal requirement for superlinear convergence of interior point algorithms, as investigated for example in \cite{Monteiro8}.

%
%


\begin{assumption}\label{ass:strictcomplementarity}
There exists an optimal solution $(X^\ast, y^\ast, Y^\ast)$ to the primal-dual SDP pair (\ref{primalSDP})-(\ref{dualSDP}) with $C = 0$ such that $X^\ast + Y^\ast \in S^n_{++}$.
\end{assumption}

\noindent A consequence of the above assumption on the homogeneous feasibility model (\ref{homogeneous1LSDFP})-(\ref{homogeneous4LSDFP}) is that it has a solution $(X^\ast, y^\ast, Y^\ast, \tau^\ast, \kappa^\ast)$ with $X^\ast + Y^\ast \in S^n_{++}$ and $\tau^\ast + \kappa^\ast > 0$.  This then implies that its SDLCP representation, {\textcolor{black}{sSDLCP}}, has a solution $({\textcolor{black}{\hat{X}^{\ast}}}, {\textcolor{black}{\hat{Y}^{\ast}}})$ such that ${\textcolor{black}{\hat{X}^{\ast}}} + {\textcolor{black}{\hat{Y}^{\ast}}} \in S^{n_1}_{++}$, that is, the {\textcolor{black}{sSDLCP}} has a strictly complementary solution.  

We consider local superlinear convergence using Algorithm \ref{mainalgorithm} in the sense of
\begin{eqnarray}\label{def:superlinearconvergence}
\frac{\mu_{k+1}}{\mu_k} \rightarrow 0, \ \ {\rm{as}}\ k \rightarrow \infty.
\end{eqnarray}
Consideration of superlinear convergence in the form (\ref{def:superlinearconvergence}) is typical in the interior point literature, such as \cite{Kojima,Luo,Potra1}.  It is closely related to local convergence behavior of iterates, as studied for example in \cite{Potrasingle}.

The following result, which is the main result in this paper, ends this subsection:
\begin{theorem}\label{thm:superlinearconvergence}
Given an initial iterate $(X_0,y_0,Y_0,\tau_0,\kappa_0) \in \mathcal{N}(\beta_1, \mu_0)$ to Algorithm \ref{mainalgorithm}, with  $(y_0,Y_0) \in \Re^m \times S^n_{++}$ feasible to the {\textcolor{black}{sDSDP}}.  Iterates generated by Algorithm \ref{mainalgorithm} converge superlinearly in the sense of (\ref{def:superlinearconvergence}).
\end{theorem}
\begin{proof}
Let us show the result in the theorem by considering iterates generated by Algorithm \ref{mainalgorithm2} instead.  These iterates are related to those generated by Algorithm \ref{mainalgorithm} in a close way as shown in Proposition \ref{prop:iteratesbothalgorithms}.  We note that Algorithm \ref{mainalgorithm2} is Algorithm 4.1 in \cite{Sim1}, and Assumptions 2.1 and 3.1 in \cite{Sim1} are satisfied for the SDLCP representation, {\textcolor{black}{sSDLCP}}, of the homogeneous feasibility model (\ref{homogeneous1LSDFP})-(\ref{homogeneous4LSDFP}) (Proposition \ref{prop:SDLCPAssumption} and strict complementarity).  Hence, results in \cite{Sim1} are applicable to our SDLCP representation.  The SDLCP representation that Algorithm \ref{mainalgorithm2} is solving has the structure of an LSDFP (Remark \ref{rem:LSDFP}), and therefore Theorem 5.1\footnote{\textcolor{black}{This theorem is reproduced in our context as Theorem \ref{thm:TheoremAnotherPaper} in {\textcolor{black}{the appendix}}.}} in \cite{Sim1} can be applied on our SDLCP representation provided that  Condition (52) in the theorem {\textcolor{black}{((\ref{eq:condition2}) in {\textcolor{black}{the appendix}})}} is satisfied.  

\noindent  Our choice of initial iterate $(X_0,y_0,Y_0,\tau_0,\kappa_0)$ to Algorithm \ref{mainalgorithm} leads to an initial iterate, $({\textcolor{black}{\hat{X}}}_0, {\textcolor{black}{\hat{Y}}}_0)$, to Algorithm \ref{mainalgorithm2} that satisfies ${\textcolor{black}{\hat{\mathcal{B}}}}({\textcolor{black}{\hat{Y}}}_0) = 0$ except possibly the $(m+n+1)^{th}$ entry\footnote{Remark \ref{rem:homogeneousmodel}  {\textcolor{black}{and {\textcolor{black}{the appendix}}}}.} of ${\textcolor{black}{\hat{\mathcal{B}}}}({\textcolor{black}{\hat{Y}}}_0)$.  Therefore, Condition (52) of Theorem 5.1 in \cite{Sim1} is satisfied, and by the theorem, we have superlinear convergence in the sense that
\begin{eqnarray*}\label{larrow:convergence}
\frac{{\textcolor{black}{\hat{\mu}}}_{k+1}}{{\textcolor{black}{\hat{\mu}}}_k} \rightarrow 0, \ \ {\rm{as}}\ k \rightarrow \infty.
\end{eqnarray*}
This implies by Proposition \ref{prop:iteratesbothalgorithms}, where we have $\mu_k^1 = \mu_k$, superlinear convergence in the sense of (\ref{def:superlinearconvergence}) using Algorithm \ref{mainalgorithm} to solve the homogeneous feasibility model (\ref{homogeneous1LSDFP})-(\ref{homogeneous4LSDFP}) for the given initial iterate.
\end{proof}

%

\begin{remark}\label{rem:NT}
Similar result as Theorem \ref{thm:superlinearconvergence} also holds when the HKM search direction or the NT search direction is used in Algorithm \ref{mainalgorithm} instead of the dual HKM search direction.  The equivalent algorithm on the SDLCP representation for the NT search direction is Algorithm 1 in \cite{Sim6}.  We can then apply Theorem 4 or Corollary 1 in \cite{Sim6} to conclude superlinear convergence of iterates when the initial iterate to Algorithm \ref{mainalgorithm} with the NT search direction {\textcolor{black}{is}} from a strictly feasible solution to the {\textcolor{black}{sDSDP}}.  The process to show this is analogous to what we have discussed and we will not repeat it here again.
\end{remark}


\section{Conclusion}\label{sec:conclusion}

In this paper, we show superlinear convergence of an implementable polynomial-time infeasible predictor-corrector primal-dual path following interior point algorithm (Algorithm \ref{mainalgorithm}) on the homogeneous feasibility model of an LSDFP for a suitable choice of initial iterate to the algorithm. {\textcolor{black}{This initial iterate can be obtained efficiently and fast by solving a primal-dual SDP pair using a primal-dual path following interior point algorithm.}}

%

\bibliographystyle{plain}
\bibliography{Reference_Sim}

\def\cprime{$'$}
\begin{thebibliography}{10}

\bibitem{Alizadeh}
F.~Alizadeh, J.~A. Haeberly, and M.~Overton.
\newblock Primal-dual interior-point methods for semidefinite programming:
  convergence rates, stability and numerical results.
\newblock {\em SIAM Journal on Optimization}, 8:746--768, 1998.

\bibitem{Alzalg}
B.~Alzalg.
\newblock A primal-dual interior-point method based on various selections of
  displacement step for symmetric optimization.
\newblock {\em Computational Optimization and Applications}, 72:363--390, 2019.

\bibitem{Anderson}
E.~D. Anderson, C.~Roos, and T.~Terlaky.
\newblock On implementing a primal-dual interior-point method for conic
  quadratic optimization.
\newblock {\em Mathematical Programming, Series B}, 95:249–277, 2003.

\bibitem{Boyd2}
S.~Boyd, L.~El Ghaoui, E.~Feron, and V.~Balakrishnan.
\newblock {\em Linear Matrix Inequalities in System and Control Theory}.
\newblock SIAM Studies in Applied Mathematics. SIAM, Philadelphia, 1994.

\bibitem{Boyd}
S.~Boyd and L.~Vandenberghe.
\newblock {\em Convex Optimization}.
\newblock Cambridge University Press, 2004.

\bibitem{Dahl}
J.~Dahl and E.~D. Anderson.
\newblock A primal-dual interior-point algorithm for nonsymmetric
  exponential-cone optimization.
\newblock {\em Mathematical Programming, Series A}, 194:341–370, 2022.

\bibitem{de_Klerk}
E.~de~Klerk, T.~Terlaky, and K.~Roos.
\newblock Self-dual embeddings.
\newblock In {\em Handbook of Semidefinite Programming (Eds: H. Wolkowicz, R.
  Saigal, L. Vandenberghe)}, volume~27 of {\em International Series in
  Operations Research \& Management Science}, pages 111--138. Springer, Boston,
  2000.

\bibitem{Faybusovich}
L.~Faybusovich and C.~Zhou.
\newblock Long-step path-following algorithm for solving symmetric programming
  problems with nonlinear objective functions.
\newblock {\em Computational Optimization and Applications}, 72:769--795, 2019.

\bibitem{Helmberg}
C.~Helmberg, F.~Rendl, R.~J. Vanderbei, and H.~Wolkowicz.
\newblock An interior-point method for semidefinite programming.
\newblock {\em SIAM Journal on Optimization}, 6:342--361, 1996.

\bibitem{Kojima}
M.~Kojima, M.~Shida, and S.~Shindoh.
\newblock Local convergence of predictor-corrector infeasible-interior-point
  algorithms for {SDP}s and {SDLCP}s.
\newblock {\em Mathematical Programming, Series A}, 80:129--160, 1998.

\bibitem{Kojima2}
M.~Kojima, M.~Shida, and S.~Shindoh.
\newblock A predictor-corrector interior-point algorithm for the semidefinite
  linear complementarity problem using the {A}lizadeh-{H}aeberly-{O}verton
  search direction.
\newblock {\em SIAM Journal on Optimization}, 9:444--465, 1999.

\bibitem{Kojima1}
M.~Kojima, S.~Shindoh, and S.~Hara.
\newblock Interior-point methods for the monotone semidefinite linear
  complementarity problem in symmetric matrices.
\newblock {\em SIAM Journal on Optimization}, 7:86--125, 1997.

\bibitem{Lu3}
Z.~Lu and R.~D.~C. Monteiro.
\newblock A note on the local convergence of a predictor-corrector
  interior-point algorithm for the semidefinite linear complementarity problem
  based on the {A}lizadeh-{H}aeberly-{O}verton search direction.
\newblock {\em SIAM Journal on Optimization}, 15:1147--1154, 2005.

\bibitem{Luo}
Z.-Q. Luo, J.~F. Sturm, and S.~Zhang.
\newblock Superlinear convergence of a symmetric primal-dual path following
  algorithm for semidefinite programming.
\newblock {\em SIAM Journal on Optimization}, 8:59--81, 1998.

\bibitem{Monteiro3}
R.~D.~C. Monteiro.
\newblock Primal-dual path following algorithms for semidefinite programming.
\newblock {\em SIAM Journal on Optimization}, 7:663--678, 1997.

\bibitem{Monteiro8}
R.~D.~C. Monteiro and S.~J. Wright.
\newblock Local convergence of interior-point algorithms for degenerate
  monotone {LCP}.
\newblock {\em Computational Optimization and Applications}, 3:131--155, 1994.

\bibitem{Nesterov}
Y.~Nesterov and M.~Todd.
\newblock Self-scaled barriers and interior-point methods for convex
  programming.
\newblock {\em Mathematics of Operations Research}, 22:1--42, 1997.

\bibitem{Nesterov1}
Y.~Nesterov and M.~Todd.
\newblock Primal-dual interior-point methods for self-scaled cones.
\newblock {\em SIAM Journal on Optimization}, 8:324--364, 1998.

\bibitem{Nesterov2}
Y.~Nesterov and L.~Tun\c{c}el.
\newblock Local superlinear convergence of polynomial-time interior-point
  methods for hyperbolicity cone optimization problems.
\newblock {\em SIAM Journal on Optimization}, 26:139--170, 2016.

\bibitem{Potrasingle}
F.~A. Potra.
\newblock {Q}-superlinear convergence of the iterates in primal-dual
  interior-point methods.
\newblock {\em Mathematical Programming, Series A}, 91:99--115, 2001.

\bibitem{Potra}
F.~A. Potra and R.~Sheng.
\newblock On homogeneous interior-point algorithms for semidefinite
  programming.
\newblock {\em Optimization Methods and Software}, 9:161--184, 1998.

\bibitem{Potra2}
F.~A. Potra and R.~Sheng.
\newblock Superlinear convergence of interior-point algorithms for semidefinite
  programming.
\newblock {\em Journal of Optimization Theory and Applications}, 99:103--119,
  1998.

\bibitem{Potra1}
F.~A. Potra and R.~Sheng.
\newblock A superlinearly convergent primal-dual infeasible-interior-point
  algorithm for semidefinite programming.
\newblock {\em SIAM Journal on Optimization}, 8:1007--1028, 1998.

\bibitem{Pougkakiotis}
S.~Pougkakiotis and J.~Gondzio.
\newblock An interior point-proximal method of multipliers for linear positive
  semi-definite programming.
\newblock {\em Journal of Optimization Theory and Applications}, 192:97--129,
  2022.

\bibitem{Rigo}
P.~R. Rig\'{o} and Z.~Darvay.
\newblock Infeasible interior-point method for symmetric optimization using a
  positive-asymptotic barrier.
\newblock {\em Computational Optimization and Applications}, 71:483--508, 2018.

\bibitem{Sim1}
C.-K. Sim.
\newblock Superlinear convergence of an infeasible predictor-corrector
  path-following interior point algorithm for a semidefinite linear
  complementarity problem using the {H}elmberg-{K}ojima-{M}onteiro direction.
\newblock {\em SIAM Journal on Optimization}, 21:102--126, 2011.

\bibitem{Sim6}
C.-K. Sim.
\newblock Interior point method on semi-definite linear complementarity
  problems using the {N}esterov-{T}odd ({NT}) search direction: {P}olynomial
  complexity and local convergence.
\newblock {\em Computational Optimization and Applications}, 74:583--621, 2019.

\bibitem{Sim2}
C.-K. Sim and G.~Zhao.
\newblock Asymptotic behavior of {H}elmberg-{K}ojima-{M}onteiro ({HKM}) paths
  in interior point methods for monotone semidefinite linear complementarity
  problem: {G}eneral theory.
\newblock {\em Journal of Optimization Theory and Applications}, 137:11--25,
  2008.

\bibitem{Sturm1}
J.~F. Sturm.
\newblock Using {S}e{D}u{M}i 1.02, a {M}atlab toolbox for optimization over
  symmetric cones ({U}pdated for {V}ersion 1.05).
\newblock 1998-2001.

\bibitem{Toh1}
K.-C. Toh, M.~J. Todd, and R.~H. T\"{u}t\"{u}nc\"{u}.
\newblock On the implementation and usage of {SDPT3} - a {M}atlab software
  package for semidefinite-quadratic-linear programming, {V}ersion 4.0.
\newblock In {\em Handbook on {S}emidefinite, {C}onic and {P}olynomial
  {O}ptimization (Eds: M. F. Anjos, J. B. Lasserre)}, volume 166 of {\em
  International Series in Operations Research \& Management Science}, pages
  715--754. Springer, 2012.

\bibitem{Yamashita}
M.~Yamashita, K.~Fujisawa, M.~Fukuda, K.~Kobayashi, K.~Nakta, and M.~Nakata.
\newblock Latest developments in the {SDPA} family for solving large-scale
  {SDP}s.
\newblock In {\em Handbook on {S}emidefinite, {C}onic and {P}olynomial
  {O}ptimization (Eds: M. F. Anjos, J. B. Lasserre)}, volume 166 of {\em
  International Series in Operations Research \& Management Science}, pages
  687--713. Springer, 2012.

\bibitem{Yamashita2}
M.~Yamashita, K.~Fujisawa, and M.~Kojima.
\newblock Implementation and evaluation of {SDPA} 6.0 ({S}emi{D}efinite
  {P}rogramming {A}lgorithm 6.0).
\newblock {\em Optimization Methods and Software}, 18:491--505, 2003.

\bibitem{Zhang}
Y.~Zhang.
\newblock On extending some primal-dual interior-point algorithms from linear
  programming to semidefinite programming.
\newblock {\em SIAM Journal on Optimization}, 8:365--386, 1998.

\end{thebibliography}

\appendix
\section{Appendix}\label{sec:Appendix}

{\textcolor{black}{For the sake of completeness, let us write down Theorem 5.1 in \cite{Sim1} in our context:
\begin{theorem}\label{thm:TheoremAnotherPaper}(see Theorem 5.1 in \cite{Sim1})
Assume that there exists a solution $({\textcolor{black}{\hat{X}^{\ast}}},{\textcolor{black}{\hat{Y}^{\ast}}}) \in S^{n_1}_+ \times S^{n_1}_+$ to an LSDFP (expressed as SDLCP (\ref{SDLCP1})-(\ref{SDLCP3})) such that ${\textcolor{black}{\hat{X}^{\ast}}} + {\textcolor{black}{\hat{Y}^{\ast}}} \in S^{n_1}_{++}$.  Let $({\textcolor{black}{\hat{X}}}_k,{\textcolor{black}{\hat{Y}}}_k)$ be iterates generated by Algorithm \ref{mainalgorithm2} on the LSDFP with the initial iterate $({\textcolor{black}{\hat{X}}}_0,{\textcolor{black}{\hat{Y}}}_0)$ satisfying
\begin{eqnarray}\label{eq:condition1}
\left( \begin{array}{l}
{\rm{svec}}\left( \begin{array}{ccc}
				0 & 0  & 0 \\
				0 & (({\textcolor{black}{\hat{\mathcal{A}}}})_{j_1+j_2+1})_{22} & 0 \\
				0 & 0 & 0
				\end{array} \right)^T \\
				\vdots \\
{\rm{svec}}\left( \begin{array}{ccc}
				0 & 0 & 0 \\
				0 & (({\textcolor{black}{\hat{\mathcal{A}}}})_{m_1})_{22} & 0 \\
				0 & 0 & 0 
				\end{array} \right)^T
\end{array} \right) {\rm{svec}}({\textcolor{black}{\hat{X}}}_0) = 0 
\end{eqnarray}
or
\begin{eqnarray}\label{eq:condition2}
\left( \begin{array}{l}
{\rm{svec}}\left( \begin{array}{ccc}
				(({\textcolor{black}{\hat{\mathcal{B}}}})_{k_1+k_2+1})_{11} & 0  & (({\textcolor{black}{\hat{\mathcal{B}}}})_{k_1+k_2+1})_{13} \\
				0 & 0 & 0 \\
				(({\textcolor{black}{\hat{\mathcal{B}}}})_{k_1+k_2+1})_{13}^T & 0 & (({\textcolor{black}{\hat{\mathcal{B}}}})_{k_1+k_2+1})_{33} 
				\end{array} \right)^T \\
				\vdots \\
{\rm{svec}}\left( \begin{array}{ccc}
				(({\textcolor{black}{\hat{\mathcal{B}}}})_{\tilde{n}_1-m_1})_{11}  & 0 & (({\textcolor{black}{\hat{\mathcal{B}}}})_{\tilde{n}_1-m_1})_{13}\\
				0 & 0 & 0 \\
				(({\textcolor{black}{\hat{\mathcal{B}}}})_{\tilde{n}_1-m_1})_{13}^T & 0 & (({\textcolor{black}{\hat{\mathcal{B}}}})_{\tilde{n}_1-m_1})_{33}
				\end{array} \right)^T
\end{array} \right) {\rm{svec}}({\textcolor{black}{\hat{Y}}}_0) = 0, 
\end{eqnarray}				
where $q_i = 0$ for $i = m_1 + 1, \ldots, \tilde{n}_1$  in (\ref{SDLCP1}) when (\ref{eq:condition1}) holds, and $q_i = 0$ for  $i = 1, \ldots, m_1$ in (\ref{SDLCP1}) when (\ref{eq:condition2}) holds.  Here, $\tilde{n}_1 = n_1(n_1 + 1)/2$.  Then the iterates converge superlinearly.
\end{theorem}}}

{\textcolor{black}{To put things in perspective, under the assumption that  there exists a solution $({\textcolor{black}{\hat{X}^{\ast}}}, {\textcolor{black}{\hat{Y}^{\ast}}}) \in S^{n_1}_+ \times S^{n_1}_+$ to the LSDFP (expressed as SDLCP (\ref{SDLCP1})-(\ref{SDLCP3})) such that ${\textcolor{black}{\hat{X}^{\ast}}} + {\textcolor{black}{\hat{Y}^{\ast}}} \in S^{n_1}_{++}$, we can assume that ${\textcolor{black}{\hat{X}^{\ast}}}$ and ${\textcolor{black}{\hat{Y}^{\ast}}}$ are given by
\begin{eqnarray*}
{\textcolor{black}{\hat{X}^{\ast}}} = \left( \begin{array}{ccc}
				\Lambda^\ast_{11} & 0 & 0 \\
				0 & 0 & 0 \\
				0 & 0 & \lambda_{k_0}^\ast
				\end{array} \right), \quad  
{\textcolor{black}{\hat{Y}^{\ast}}} = \left( \begin{array}{ccc}
				0 & 0  & 0 \\
				0 & \Lambda^\ast_{22} & 0 \\
				0 & 0 & 0
				\end{array} \right),
\end{eqnarray*} }}
{\textcolor{black}{where $\Lambda^\ast_{11} = {\rm{Diag}}(\lambda_1^\ast, \ldots, \lambda^\ast_{k_0-1}) \in S^{k_0-1}_{++}$ and $\Lambda^\ast_{22} = {\rm{Diag}}(\lambda_{k_0+1}^\ast, \ldots, \lambda^\ast_{n_1}) \in S^{n_1 - k_0}_{++}$.  Here, $\lambda_1^\ast, \ldots, \lambda_{n_1}^\ast$ are real numbers greater than zero.  In this appendix, we always partition a matrix $S \in S^{n_1}$ in the way ${\textcolor{black}{\hat{X}^{\ast}}}$ and ${\textcolor{black}{\hat{Y}^{\ast}}}$ are partitioned, i.e., $S$ is partitioned as $\left(\begin{array}{ccc}
					S_{11} & S_{12} & S_{13} \\
					S_{12}^T & S_{22} & S_{23} \\
					S_{13}^T & S_{23}^T & S_{33}
					\end{array} \right)$, where $S_{11} \in S^{k_0-1}, S_{22} \in S^{n_1 - k_0}, S_{33} \in \Re$ and $S_{12} \in \Re^{(k_0 - 1)\times (n_1 - k_0)}, S_{13} \in \Re^{k_0-1}, S_{23} \in \Re^{n_1 - k_0}$.   Using this partition, we perform block Gaussian eliminations and row operations on ${\textcolor{black}{\hat{\mathcal{A}}}}, {\textcolor{black}{\hat{\mathcal{B}}}}: S^{n_1} \rightarrow \Re^{\tilde{n}_1}$ in (\ref{SDLCP1}) so that they can be written in matrix form as:
\begin{eqnarray*}
{\textcolor{black}{\hat{\mathcal{A}}}} = \left( \begin{array}{c}
						{\textcolor{black}{\bar{\mathcal{A}}}} \\
						0
					\end{array} \right), \quad {\textcolor{black}{\hat{\mathcal{B}}}} = \left( \begin{array}{c}
																		0 \\
																	{\textcolor{black}{\bar{\mathcal{B}}}}
																	\end{array} \right),
\end{eqnarray*}}}
{\textcolor{black}{where
\begin{eqnarray*}
{\textcolor{black}{\bar{\mathcal{A}}}} = \left( \begin{array}{l}
{\rm{svec}}\left( \begin{array}{ccc}
				(({\textcolor{black}{\hat{\mathcal{A}}}})_1)_{11} & (({\textcolor{black}{\hat{\mathcal{A}}}})_1)_{12} & (({\textcolor{black}{\hat{\mathcal{A}}}})_1)_{13} \\
				(({\textcolor{black}{\hat{\mathcal{A}}}})_1)_{12}^T & (({\textcolor{black}{\hat{\mathcal{A}}}})_1)_{22} & (({\textcolor{black}{\hat{\mathcal{A}}}})_1)_{23} \\
				(({\textcolor{black}{\hat{\mathcal{A}}}})_1)_{13}^T & (({\textcolor{black}{\hat{\mathcal{A}}}})_1)_{23}^T & (({\textcolor{black}{\hat{\mathcal{A}}}})_1)_{33}
				\end{array} \right)^T \\
				\vdots \\
{\rm{svec}}\left( \begin{array}{ccc}
				(({\textcolor{black}{\hat{\mathcal{A}}}})_{j_1})_{11} & (({\textcolor{black}{\hat{\mathcal{A}}}})_{j_1})_{12} & (({\textcolor{black}{\hat{\mathcal{A}}}})_{j_1})_{13} \\
				(({\textcolor{black}{\hat{\mathcal{A}}}})_{j_1})_{12}^T & (({\textcolor{black}{\hat{\mathcal{A}}}})_{j_1})_{22} & (({\textcolor{black}{\hat{\mathcal{A}}}})_{j_1})_{23} \\
				(({\textcolor{black}{\hat{\mathcal{A}}}})_{j_1})_{13}^T & (({\textcolor{black}{\hat{\mathcal{A}}}})_{j_1})_{23}^T & (({\textcolor{black}{\hat{\mathcal{A}}}})_{j_1})_{33}
				\end{array} \right)^T \\
{\rm{svec}}\left( \begin{array}{ccc}
				0 & (({\textcolor{black}{\hat{\mathcal{A}}}})_{j_1+1})_{12} & 0 \\
				(({\textcolor{black}{\hat{\mathcal{A}}}})_{j_1+1})_{12}^T & (({\textcolor{black}{\hat{\mathcal{A}}}})_{j_1+1})_{22} & (({\textcolor{black}{\hat{\mathcal{A}}}})_{j_1+1})_{23} \\
				0 & (({\textcolor{black}{\hat{\mathcal{A}}}})_{j_1+1})_{23}^T & 0
				\end{array} \right)^T \\
				\vdots \\
{\rm{svec}}\left( \begin{array}{ccc}
				0 & (({\textcolor{black}{\hat{\mathcal{A}}}})_{j_1+j_2})_{12} & 0 \\
				(({\textcolor{black}{\hat{\mathcal{A}}}})_{j_1+j_2})_{12}^T & (({\textcolor{black}{\hat{\mathcal{A}}}})_{j_1+j_2})_{22} & (({\textcolor{black}{\hat{\mathcal{A}}}})_{j_1+j_2})_{23} \\
				0 & (({\textcolor{black}{\hat{\mathcal{A}}}})_{j_1+j_2})_{23}^T & 0
				\end{array} \right)^T \\
{\rm{svec}}\left( \begin{array}{ccc}
				0 & 0 & 0 \\
				0 & (({\textcolor{black}{\hat{\mathcal{A}}}})_{j_1+j_2+1})_{22} & 0 \\
				0 & 0 & 0
				\end{array} \right)^T \\
\vdots \\
{\rm{svec}}\left( \begin{array}{ccc}
				0 & 0 & 0 \\
				0 & (({\textcolor{black}{\hat{\mathcal{A}}}})_{m_1})_{22} & 0 \\
				0 & 0 & 0
				\end{array} \right)^T
\end{array} \right)
\end{eqnarray*}}}
{\textcolor{black}{and
\begin{eqnarray*}
{\textcolor{black}{\bar{\mathcal{B}}}} = \left( \begin{array}{l}
{\rm{svec}}\left( \begin{array}{ccc}
				(({\textcolor{black}{\hat{\mathcal{B}}}})_1)_{11} & (({\textcolor{black}{\hat{\mathcal{B}}}})_1)_{12} & (({\textcolor{black}{\hat{\mathcal{B}}}})_1)_{13} \\
				(({\textcolor{black}{\hat{\mathcal{B}}}})_1)_{12}^T & (({\textcolor{black}{\hat{\mathcal{B}}}})_1)_{22} & (({\textcolor{black}{\hat{\mathcal{B}}}})_1)_{23} \\
				(({\textcolor{black}{\hat{\mathcal{B}}}})_1)_{13}^T & (({\textcolor{black}{\hat{\mathcal{B}}}})_1)_{23}^T & (({\textcolor{black}{\hat{\mathcal{B}}}})_1)_{33}
				\end{array} \right)^T \\
				\vdots \\
{\rm{svec}}\left( \begin{array}{ccc}
				(({\textcolor{black}{\hat{\mathcal{B}}}})_{k_1})_{11} & (({\textcolor{black}{\hat{\mathcal{B}}}})_{k_1})_{12} & (({\textcolor{black}{\hat{\mathcal{B}}}})_{k_1})_{13} \\
				(({\textcolor{black}{\hat{\mathcal{B}}}})_{k_1})_{12}^T & (({\textcolor{black}{\hat{\mathcal{B}}}})_{k_1})_{22} & (({\textcolor{black}{\hat{\mathcal{B}}}})_{k_1})_{23} \\
				(({\textcolor{black}{\hat{\mathcal{B}}}})_{k_1})_{13}^T & (({\textcolor{black}{\hat{\mathcal{B}}}})_{k_1})_{23}^T & (({\textcolor{black}{\hat{\mathcal{B}}}})_{k_1})_{33}
				\end{array} \right)^T \\
{\rm{svec}}\left( \begin{array}{ccc}
				(({\textcolor{black}{\hat{\mathcal{B}}}})_{k_1+1})_{11} & (({\textcolor{black}{\hat{\mathcal{B}}}})_{k_1+1})_{12} & (({\textcolor{black}{\hat{\mathcal{B}}}})_{k_1+1})_{13} \\
				(({\textcolor{black}{\hat{\mathcal{B}}}})_{k_1+1})_{12}^T & 0 & (({\textcolor{black}{\hat{\mathcal{B}}}})_{k_1+1})_{23} \\
				(({\textcolor{black}{\hat{\mathcal{B}}}})_{k_1+1})_{13}^T & (({\textcolor{black}{\hat{\mathcal{B}}}})_{k_1+1})_{23}^T & (({\textcolor{black}{\hat{\mathcal{B}}}})_{k_1+1})_{33}
				\end{array} \right)^T \\
				\vdots \\
{\rm{svec}}\left( \begin{array}{ccc}
				(({\textcolor{black}{\hat{\mathcal{B}}}})_{k_1+k_2})_{11} & (({\textcolor{black}{\hat{\mathcal{B}}}})_{k_1+k_2})_{12} & (({\textcolor{black}{\hat{\mathcal{B}}}})_{k_1+k_2})_{13} \\
				(({\textcolor{black}{\hat{\mathcal{B}}}})_{k_1+k_2})_{12}^T & 0 & (({\textcolor{black}{\hat{\mathcal{B}}}})_{k_1+k_2})_{23} \\
				(({\textcolor{black}{\hat{\mathcal{B}}}})_{k_1+k_2})_{13}^T & (({\textcolor{black}{\hat{\mathcal{B}}}})_{k_1+k_2})_{23}^T & (({\textcolor{black}{\hat{\mathcal{B}}}})_{k_1+k_2})_{33}
				\end{array} \right)^T \\
{\rm{svec}}\left( \begin{array}{ccc}
				(({\textcolor{black}{\hat{\mathcal{B}}}})_{k_1+k_2+1})_{11} & 0 & (({\textcolor{black}{\hat{\mathcal{B}}}})_{k_1+k_2+1})_{13} \\
				0 & 0 & 0 \\
				(({\textcolor{black}{\hat{\mathcal{B}}}})_{k_1+k_2+1})_{13}^T & 0 & (({\textcolor{black}{\hat{\mathcal{B}}}})_{k_1+k_2+1})_{33}
				\end{array} \right)^T \\
\vdots \\
{\rm{svec}}\left( \begin{array}{ccc}
				(({\textcolor{black}{\hat{\mathcal{B}}}})_{\tilde{n}_1 - m_1})_{11} & 0 & (({\textcolor{black}{\hat{\mathcal{B}}}})_{\tilde{n}_1 - m_1})_{13} \\
				0 & 0 & 0 \\
				(({\textcolor{black}{\hat{\mathcal{B}}}})_{\tilde{n}_1 - m_1})_{13}^T & 0 & (({\textcolor{black}{\hat{\mathcal{B}}}})_{\tilde{n}_1 - m_1})_{33}
				\end{array} \right)^T
\end{array} \right).
\end{eqnarray*}}}
{\textcolor{black}{Details on the above reformulation of ${\textcolor{black}{\hat{\mathcal{A}}}}$ and ${\textcolor{black}{\hat{\mathcal{B}}}}$ can be found in \cite{Sim1,Sim6,Sim2}.}}

{\textcolor{black}{In the context of our SDLCP representation, {\textcolor{black}{sSDLCP}}, of the homogeneous feasibility model of the {\textcolor{black}{sLSDFP}}, {\textcolor{black}{where we have $m_1 = m + n, n_1 = n+1$, and {\textcolor{black}{$\hat{\mathcal{A}}, \hat{\mathcal{B}}$}} are given by (\ref{def:A1}), (\ref{def:B1}) respectively}}, we choose $B_1$ such that $(B_1)_{22} \not= 0$.  In this way, when we perform block Gaussian eliminations and row operations on ${\textcolor{black}{\hat{\mathcal{B}}}}$, its final position with respect to other rows in ${\textcolor{black}{\bar{\mathcal{B}}}}$ remains unchanged, that is, ${\rm{svec}}\left( \begin{array}{cc}
																	B_1 & 0 \\
																	0 & d_1 
																	\end{array} \right)^T$ remains in the first row of ${\textcolor{black}{\bar{\mathcal{B}}}}$ or the $(m_1 + 1)^{th}$ row of ${\textcolor{black}{\hat{\mathcal{B}}}}$.  We see that (\ref{eq:condition2}) {\textcolor{black}{(which is (52) of Theorem 5.1 in \cite{Sim1})}} then holds when ${\textcolor{black}{\hat{Y}}}_0$ is from a strictly feasible solution to the {\textcolor{black}{sDSDP}}.}}

\end{document}